\documentclass[12pt]{amsart} 
\usepackage{amssymb,amsfonts,amsthm,mathrsfs,graphics}
\usepackage{palatino}
\usepackage{natbib}
\def\bibfont{\normalsize}

\theoremstyle{plain}
\newtheorem{theorem}{Theorem}[section]
\newtheorem{conj}[theorem]{Conjecture}
\newtheorem{defi}[theorem]{Definition}
\newtheorem{exa}[theorem]{\it Example}
\newtheorem{lem}[theorem]{Lemma}
\newtheorem{prop}[theorem]{Proposition}
\newtheorem{rem}[theorem]{\it Remark}
\newcommand{\sect}[1]{\setcounter{equation}0\section{#1}}

\def\WH{W\times H}
\def\etal{\emph{et al.\ }}
\def\rec{1/\kern-2pt\sqrt3}

\def\tt{t}

\def\j{\\[-.7pt]}
\def\eqref#1{(\ref{#1})}
\def\ge{\geqslant}
\def\le{\leqslant}
\def\Re{\mathop{\mathrm{Re}}}
\def\y{\\[4pt]}
\def\fr#1#2{\hbox{\Large$\frac{#1}{#2}$}}
\def\ha{\frac12}

\def\sa{\frac16}
\def\oz{\overline{z}}
\def\cd{\ \hbox{\footnotesize$\bullet$}\ }
\def\tr{\mathop{\mathrm{tr}}}
\def\SO{\mathrm{SO}}
\def\SU{\mathrm{SU}}
\def\su{\mathfrak{su}}
\def\U{\mathrm{U}}
\def\alp{\alpha}

\def\Ga{\Gamma}
\def\om{\omega}
\def\si{\sigma}
\def\th{\theta}
\def\ds{\displaystyle}
\def\ts{\textstyle}
\def\ba{\begin{array}}
\def\ea{\end{array}}
\def\+{\!+\!}
\def\-{\!-\!}
\def\lra{\longrightarrow}
\def\ol{\overline}
\def\C{\mathbb{C}}
\def\CP{\mathbb{CP}}
\def\sH{\mathscr{H}}
\def\HP{\mathbb{HP}}
\def\LL{\mathbb{L}}
\def\R{\mathbb{R}}
\def\SS{\mathbb{S}}
\def\SSt{\mathbb{S}_\th}
\def\Z{\mathbb{Z}}
\def\uu{\mathbf{u}}
\def\vv{\mathbf{v}}
\def\w{\mathbf{w}}

\def\yy{\mathbf{y}}
\def\z{\mathbf{z}}
\def\sC{\mathscr{C}}
\def\sR{\mathscr{R}}
\def\sT{\mathscr{T}}

\begin{document}
\large

\parskip2pt
\parindent20pt
\mathsurround1pt

\def\tif#1{\fontsize{20}{24}\selectfont\sc #1}
\def\auf#1{\fontsize{16}{18}\selectfont\sc #1}
\def\add#1{\fontsize{12}{16}\selectfont\it #1}

\begin{center}

{\tif Surveying points in the complex projective plane}

\vskip15pt

{\auf Lane Hughston and Simon Salamon}


\vskip25pt

\begin{quote}\small
We classify SIC-POVMs of rank one in $\CP^2,$ or equivalently sets of
nine equally-spaced points in $\CP^2,$ without the assumption of group
covariance. If two points are fixed, the remaining seven must lie on a
pinched torus that a standard moment mapping projects to a circle in
$\R^3.$ We use this approach to prove that any SIC set in $\CP^2$ is
isometric to a known solution, given by nine points lying in
triples on the equators of the three 2-spheres each defined by the
vanishing of one homogeneous coordinate. We set up a system of
equations to describe hexagons in $\CP^2$ with the property that any
two vertices are related by a cross ratio (transition probability) of
$1/4.$ We then symmetrize the equations, factor out by the known
solutions, and compute a Gr\"obner basis to show that no SIC sets
remain. We do find new configurations of nine points in which 27 of
the 36 pairs of vertices of the configuration are equally spaced.
\end{quote}

\end{center}

\vskip10pt

\section*{Introduction}

A symmetric, informationally complete, positive-operator valued
measure or SIC-POVM on the Hermitian vector space $\C^n$ is a set
$\{P_j\}$ of $n^2$ rank-one projection operators such that
\begin{equation*}
\frac1n\!\sum_{i=1}^{n^2} P_i = I,
\end{equation*}
and 
\begin{equation*}
\tr(P_jP_k)=\frac1{n+1}(n\delta_{jk}+1)
\end{equation*}
for all $j,k.$ Such objects attracted wide attention following
conjectures about their existence made by Zauner \cite{Z} in 1999 and
Renes \etal \cite{RSC} in 2004, and since then have been investigated
by a large number of authors, along with higher rank versions and the
allied concept of mutually unbiased basis. See, for example,
\cite{App,AFZ,AYZ,DBBA,Durt,Gour,Grl,Hogg,Hugh,SG,Woo,Zhu}, and
references cited therein.

SIC-POVMs arise in the theory of quantum measurement (see Davies
\cite{Dav} and Holevo \cite{Hol} for the significance of general
POVMs), and are of great interest in connection with their potential
applications to quantum tomography. The idea is the following. Suppose
that one has a large number of independent identical copies of a
quantum system (say, a large molecule), the state (or `structure') of
which is unknown and needs to be determined. A SIC-POVM can be thought
of as a kind of symmetrically oriented machine that can be used to
make a single tomographic measurement on each independent copy of the
molecule, with the property that once the results of the various
measurements have been gathered for a sufficiently large number of
molecules, the state of the molecule can be efficiently determined to
a high degree of accuracy.  The `symmetric orientation' is not with
respect to ordinary three-dimensional physical space (as in the
classical tomography of medical imaging), but rather with respect to
the space of pure quantum states.

Since each element $P_i$ of a SIC-POVM is a matrix of rank one and
trace unity, it determines a point in complex projective space
$\CP^{n-1}.$ It is well known that a SIC-POVM can then be defined as a
configuration of $n^2$ points in $\CP^{n-1}$ that are mutually
equidistant under the standard K\"ahler metric \cite{LS,Wel}. This is
the definition that we shall adopt in \S\ref{HW}, and the distance is
determined by Lemma~\ref{CSI}. Such a set of points is often called a
`SIC', but we favour the expression `SIC set'.

The existence of such configurations (for example, nine equidistant
points in $\CP^2,$ or sixteen equidistant points in $\CP^3$) is
counterintuitive to our everyday way of thinking in which a regular
simplex in $\R^n$ has $n+1$ vertices (but see \cite{GKMS}). It has
been conjectured that $\CP^{n-1}$ possesses such a configuration for
every $n$ \cite{RSC,Z}. There is evidence for this for $n$ up to at
least $67,$ and various explicit solutions have been found in lower
dimensions. Most of the known SIC sets in higher dimensions are
constructed as orbits of a Heisenberg group $\WH$ acting on
$\CP^{n-1}$ (see Section~\ref{HW}), and representative vectors occur
as eigenvectors of an isometry that is an outer automorphism of $\WH.$
In the case $n=5$, the automorphisms of $\WH$ play a key role in the
construction of the celebrated Horrocks-Mumford bundle over $\CP^4$ in
\cite{HM}, which is an excellent reference for this group theory. In
the case $n=3$ (and more generally, when $n$ is prime) any finite
group of isometries whose orbit is a SIC set must be conjugate to
$\WH$ \cite{Zhu}, but in this paper we work without the assumption of
group covariance (see Grassl \cite{Grl}).

The space $\CP^1,$ endowed with the Fubini-Study metric, is isometric
to the standard two-sphere, and embedding this in $\R^3$ is a simple
example of the representation of $\CP^{n-1}$ as an adjoint orbit in
the Lie algebra $\su(n)$ of its isometry group. The existence of a SIC
set can then be interpreted as a statement about the placement of such
orbits. The problem can also be formulated so as to apply to more
general (co-)adjoint orbits in a Lie algebra.

The vertices of any inscribed regular tetrahedron in $S^2$ provide a
SIC set for $\CP^1$ ($n=2$). The situation for the projective plane
$\CP^2$ is already surprisingly intricate, and the case $n=3$ is
characterized by the existence of continuous families of non-congruent
SIC sets. It is easy to begin their study. Using homogeneous
coordinates, any three equally-spaced points on the equator
$\{[0,z_2,z_3]:|z_2|=|z_3|\}$ of the two-sphere $z_1=0$ lie in a SIC
set formed by adding three equally-spaced points from each of the
equators of the two-spheres $z_2=0$ and $z_3=0.$ If the diameter of
$\CP^2$ is chosen to be $\pi,$ all nine points are a distance $2\pi/3$
apart. Moreover, if the three triples match up so as to lie on a total
of twelve projective lines, the nine points are the flexes of a plane
cubic curve \cite{Hugh}.

In this paper, we show that any SIC set in $\CP^2$ is congruent to one
of those just described (see Theorem~\ref{strong}). This result will
not surprise the experts; it has perhaps been verified numerically,
and is apparently a consequence of computer-aided results in
\cite{Sz}. Our proof relies on a computation for its final step but is
predominantly analytical. We use the two-point homogeneity of $\CP^2$
to fix two points of a SIC set; applying the moment mapping relative
to a maximal torus shows that the remaining seven points lie in a
pinched torus above a circle $\sC$ in $\R^3$ (illustrated in
Figure~1). We exhibit the known solutions in a different form
(Proposition~\ref{7}) and characterize them by a symmetry condition
(Lemma~\ref{LR}). Adding three more points a distance $2\pi/3$ from
the first two and from each other leads to a polynomial equation that
is symmetric in three variables $x,y,z$ that represent the tangents of
angles measured around $\sC$ (Theorem~\ref{fa}). The resulting
geometry is illustrated in Section~\ref{Gi}.

Adding a sixth point allows us to write down four equations in four
variables $t,x,y,z.$ When these are totally symmetrized, we obtain a
system that represents a necessary condition for the six points to
form part of a SIC set. For the known solutions, at least one of the
four points on the pinched torus must project to $\sC$ with an angle
equal to $\pm\pi/6.$ This fact enables us to focus attention on the
so-called quotient ideal that parametrizes `extra' solutions, and to
describe it by means of an appropriate Gr\"obner basis. Once one root
$t$ is fixed, the extra solutions form a finite set and the final step
is to determine its size. There are too few extra solutions for these
to arise from an undiscovered SIC set.

This paper had its origins in a number of survey talks aimed at
bringing elements of the SIC-POVM problem in various low dimensions to
the attention of a wider audience, and the title and figures reflect
this.  We focus on the case $n=3$ from Section~\ref{mt} onwards, and
Sections~6--10 contain the more specialized material required to
achieve our goal. The Fubini-Study metric on an ambient projective
space plays a central role in the construction or approximation of
K\"ahler-Einstein metrics on algebraic varieties, and it is our hope
that more general theory may shed further light on the discrete
problem outlined above.

\sect{Hermitian preliminaries}

\noindent We begin with a few remarks to fix conventions. The complex
vector space
\begin{equation} 
\C^n=\{\z=(z_1,\ldots z_n)^\top:z_i\in\C\}
\end{equation}
of column vectors comes equipped with a Hermitian form
\begin{equation}\left<\w,\z\right> = \left<\w|\,\z\right> = 
\sum_{i=1}^n\ol w_iz_i\end{equation}
which is anti-linear in the first (bra) position. Each fixed $\w$
defines a linear functional $\z\mapsto\left<\w,\z\right>,$ and
\begin{equation} 
\w\ \mapsto\ \left<\w,\cd\right>
\end{equation} is an anti-linear
bijective mapping $h\colon V \to V^*,$ equivalently an
\emph{isomorphism} $V\cong\ol V^*$ of complex vector spaces.  Complex
projective space is the quotient
\begin{equation} 
\CP^{n-1}=\frac{\C^n\setminus 0}{\C^*},
\end{equation} 
consisting of one-dimensional subspaces of $\C^n$ or \emph{rays}, and
is a compact topological space. For any non-zero $\w\in\C^n$ the
associated point in $\CP^{n-1}$ will be denoted by $[\w].$ Each such
point determines a conjugate hyperplane $W$ defined by
\begin{equation} 
W=\mathbb{P}(\ker h(\w))\cong\CP^{n-2}\subset\CP^{n-1}.
\end{equation} This is
the geometrical content of the map $h.$ 

Two points $[\w],[\z]$ lie on a unique projective line
$\LL\cong\CP^1.$ The associated conjugate hyperplanes $W,Z$ intersect
  $\LL$ in $[\w'],[\z'],$ where
\begin{equation}
\w' = \left<\w,\z\right>\w-\left<\w,\w\right>\z,
\qquad
\z' = \left<\z,\z\right>\w-\left<\z,\w\right>\z.
\end{equation} 
The resulting four points, taken in the order $[\w],[\z],[\z'],[\w'],$
have inhomogeneous coordinates
\begin{equation} \infty,\quad 0,\quad
-\big<\z,\z\big>/\big<\z,\w\big>,\quad
-\left<\w,\z\right>/\big<\w,\w\big>,\end{equation}
and a real cross ratio
\begin{equation} 
\kappa([\w],[\z]) = 
\frac{\big<\w,\z\big>\big<\z,\w\big>}{\big<\w,\w\big>\big<\z,\z\big>}
= \frac{|\big<\w,\z\big>|^2}{\|\w\|^2\|\z\|^2}\in[0,1].
\end{equation} 
When the points of $\CP^{n-1}$ are interpreted as pure quantum states,
$\kappa$ can be regarded as a transition probability \cite{AshS,
  BH,Gib, Hugh2,Hugh3}. The Fubini-Study distance $d$ between the
points $[\w]$ and $[\z]$ is defined by expressing the cross ratio as
$\cos^2(d/2),$ so that
\begin{equation}\label{FS}
d([\w],[\z]) = 
2\arccos\left(\frac{|\big<\w,\z\big>|}{\|\w\|\|\z\|}\right)\in[0,\pi].
\end{equation}
When $n=2$ we get $\CP^1\cong S^2.$ We shall see in Example~\ref{abc}
that $d$ is the spherical distance
\begin{equation}
\th=\arccos\big|\big<\uu,\vv\big>\big|,\qquad \uu,\vv\in S^2,
\end{equation}
measuring the arclength of a great circle joining $\uu$ and $\vv.$ The
$\CP^1$ calculation confirms that $d$ is the usual distance measured
along geodesics of $\CP^{n-1}$ since any two points of the latter lie
on a unique projective line $\CP^1.$ The distance \eqref{FS} satisfies
the triangle inequality
\begin{equation} 
d([\w],[\z])\le d([\w],[\yy])+d([\yy],[\z]).
\end{equation} 
This can be verified by working inside the $\CP^2$ that contains
$[\w],$ $[\yy],$ $[\z].$

The so-called Fubini-Study metric is the square $ds^2$ of the
infinitesimal distance between $[\z]$ and $[\z+d\z],$ computed using
\begin{equation}\label{taylor}
\ba{rcl}
 \kappa([\z],[\z+d\z]) &=& \ds
\frac{\|\z\|^2+2\Re\big<\z,d\z\big>+|\big<\z,d\z\big>|^2/\|\z\|^2}
{\|\z\|^2+2\Re\big<\z,d\z\big>+\|d\z\|^2}\\[20pt]
&=& \ds
1-\frac{\|d\z\|^2}{\|\z\|^2}+\frac{|\big<\z,d\z\big>|^2}{\|\z\|^4}
+ O(\|d\z\|^3).
\ea
\end{equation}
There are no first-order terms, and we obtain the Riemannian metric
$g=ds^2$ where
\begin{equation} 
ds^2\ =\ 4\frac{\|\z\|^2\|d\z\|^2-|\big<\z,d\z\big>|^2}{\|\z\|^4}.
\end{equation}
If we set $z_n=1,$ and use the summation convention over the remaining
indices $z_1,\ldots,z_{n-1},$ then in the traditional notation we have
\begin{equation} 
g_{\alp\beta}dz^\alp d\ol z^\beta\ =\ 
4\frac{(\ol z_\alp z^\alp+1)dz_\beta d\ol z^\beta-
\ol z_\alp z_\beta dz^\alp d\ol z^\beta}{(\ol z_\alp z^\alp+1)^2}.
\end{equation}
See, for example, Arnold \cite{Arn} and Kobayashi and Nomizu
\cite{KN}. When $n=2,$ we obtain the classical first fundamental form
\begin{equation} 
ds^2\ =\ \frac{4\,dz d\ol z}{(1+|z|^2)^2}\ =\
\frac{8(dx^2+dy^2)}{(1+x^2+y^2)^2}
\end{equation} 
on the two-sphere $S^2,$ in which $x,y$ are isothermal coordinates.

\sect{The special unitary group}

\noindent The Hermitian form $h$ is invariant under the action of the
unitary group
\begin{equation} 
\U(n)=\{X\in\C^{n,n}:\ol X\!X^\top=I\}.
\end{equation}
Its centre consists of scalar multiples $e^{it}I$ that act trivially
on $\CP^{n-1}.$ So we consider the special unitary group
\begin{equation} 
\SU(n)=\{X\in\U(n):\det X=1\},
\end{equation}
whose centre is $\Z_n=\big<e^{2\pi i/n}I\big>.$ The next result is due
to Wigner \cite{Wig}; a modern treatment is given in \cite{Fr}.

\begin{theorem} 
The isometry group of the Fubini-Study space $\CP^{n-1},$ i.e.\ the
group of bijections preserving the distance $d,$ is generated by
$\SU(n)/\Z_n$ and $[\z]\mapsto[\ol\z].$
\end{theorem} 

The Lie algebra $\su(n)$ can (as a vector space) be defined as the
tangent space $T_I\SU(n)$ at the identity. It consists of tangent
vectors $A=\dot X_0$ to curves $X_t=I+tA+O(t^2)$ in $U(n).$ Thus
\begin{equation}
\su(n)=\{A\in \C^{n,n}:\ol A+A^\top=0,\ \tr A=0\}.
\end{equation}
A matrix $M\in\SU(n)$ acts on $\su(n)$ by the adjoint representation
\begin{equation} 
A\mapsto M\!AM^{-1}=M\!AM^\top.
\end{equation}
The space $\su(n)$ carries an invariant inner product
\begin{equation} \big<A,B\big>=-\tr(AB),\end{equation}
and $\SU(n)$ itself carries a bi-invariant Riemannian invariant.
We shall work with the corresponding affine space
\begin{equation}
\sH_n=\{A\in \C^{n,n}: \ol A=A^\top,\ \tr A=1\}
\end{equation}
of Hermitian matrices of trace one. There is an obvious bijection
\begin{equation}\label{obv}
\sH_n\ \stackrel\cong\lra\ \su(n),
\end{equation} 
given by $A\mapsto i(A-n^{-1}I).$

\bigbreak

The \emph{canonical embedding} of $\CP^{n-1}$ into $\sH_n$ is a
variant of the moment mapping for the adjoint action of $\SU(n).$ To
describe it, assume for convenience that all vectors are
normalized. Thus, we set $\|\z\|=1$ ($\z\in S^{2n-1}$) and there
remains only a phase ambiguity in passing to a point $[\z]=[e^{it}\z]$
of $\CP^{n-1}.$ Map $[\z]$ to

\begin{equation}\label{can}
P_\z = \z\ol\z^\top = \ts
\left(\ba{l}
|z_1|^2 \quad z_1\ol z_2 \quad z_1\ol z_3\cdots\\
z_2\ol z_1 \quad |z_z|^2 \quad z_2\ol z_3\cdots\\
z_3\ol z_1\quad \cdots\\
\cdots\quad \cdots\ea\right),
\end{equation}
which is a projection operator (meaning $P^2=P$) of rank one. 
The injective map 
\begin{equation}\label{i}
i\colon\CP^{n-1}\hookrightarrow\sH_n
\end{equation} 
defined by $[\z]\mapsto P_\z$ is $\SU(n)$-equivariant. We can use it
to measure distances since
\begin{equation}
\kappa([\w],[\z]) = \big|\big<\w,\z\big>\big|^2 = \tr(P_\w P_\z),
\end{equation}
assuming $\|\z\|=1=\|\w\|.$ Moreover, the derivative
\begin{equation}
i_*\colon T_x\CP^{n-1}\hookrightarrow T_x\sH_n\cong\R^N
\end{equation} 
is $\U(n-1)$-equivariant, and \eqref{i} is an isometric embedding.

\begin{exa} [The Bloch sphere] \label{abc} \rm
For $n=2,$ the image of this map consists of
the matrices
\begin{equation} \left(\!\ba{cc} 
|z_1|^2 & \ol z_1z_2\\[3pt] z_1\ol z_2 & |z_2|^2\ea\!\right) =
\fr12\!\left(\!\ba{cc} 1+a & b+ic\\[2pt] b-ic &
1-a\ea\!\right)\end{equation} with $|z_1|^2+|z_2|^2=1$ and
$a^2+b^2+c^2=1.$ This provides the well-known isomorphism $\CP^1\cong
S^2.$ The angle $\th$ between two unit vectors in $\R^3$ is given by
\begin{equation} 
aa'+bb'+cc'=\cos\th.
\end{equation}
The inner product in $\sH_2$ is then
\begin{equation} 
\ts \fr12(1+aa'+bb'+cc')=\fr12(1+\cos\th)=\cos^2(\th/2).
\end{equation}
But $\th$ is also the standard distance $d$ along the great circle on
the surface of the sphere joining the endpoints of the two unit
vectors.
\end{exa}

In the example above, fix (say) the north pole $p\in S^2,$ and
consider the function $\kappa_p=\sin^2(\th/2)$ where $\th$ is now
latitude in radians. Its gradient $\nabla\kappa_p$ is tangent to the
meridians joining $p$ to the south pole $p',$ whereas
$J(\nabla\kappa_p)$ is a vector field that represents rotation about
$pp'.$ This situation is generalized to higher dimensions as follows.
The composition
\begin{equation}
\CP^{n-1}\ \lra\ \sH_n\ \stackrel\cong\lra\ \su(n),
\end{equation} 
where $[\z]\mapsto i(P_\z-n^{-1}I),$ is a \emph{moment mapping} of the
type determined whenever a Lie group acts on a symplectic
manifold. The image (isomorphic to $\CP^{n-1}$) inside $\su(n)$ is an
orbit for the action of $\SU(n).$ Any such adjoint orbit carries a
K\"ahler metric by general principles. Fix a point
$p=[\z]\in\CP^{-1},$ and consider the function $\kappa_p$ defined by
\begin{equation} 
\kappa_p([\w])=\kappa([\z],[\w])=\tr(P_\z P_\w).
\end{equation}
We have

\begin{prop} 
The rotated gradient $J(\nabla\kappa_p)$ is the infinitesimal isometry
(Killing field) associated to $i(P_\z-n^{-1}I).$
\end{prop} 

For further details of various aspects of the K\"ahlerian geometry of
the space of pure quantum states, see Anandan and Aharonov \cite{AA},
Ashtekar and Schilling \cite{AshS}, Bengtsson and Zyczkowski
\cite{BZ}, Brody and Hughston \cite{BH}, Gibbons \cite{Gib}, Hughston
\cite{Hugh2,Hugh3}, and Kibble \cite{Kib}.

\sect{Sets of points in projective space}\label{HW}

\noindent We choose to begin with

\begin{defi} 
A SIC-POVM or SIC set is a collection $\SS$ of $n^2$ points $[\z_i]$
in $\CP^{n-1}$ that are mutually equidistant, so if $\|\z_i\|=1$ then
\begin{equation}
|\big<\z_i,\z_j\big>|^2=\kappa,\qquad i\ne j,
\end{equation}
for some fixed cross ratio $\kappa\in[0,1).$
\end{defi}

We can associate to $[\z_i]$ the point $P_i=P_{[\z_i]}$ in $\sH_n.$ A SIC set then
consists of a regular simplex embedded in 
\begin{equation} \sH_n\cong\su(n)\cong\R^N,\quad N=n^2\-1\end{equation}
with $n^2$ vertices $\{P_i\}$ that lie in the adjoint orbit
$\CP^{n-1}.$ The latter requirement is the crucial one, since a
regular simplex with $n^2$ vertices in $\R^N$ is readily obtained by
projecting an arbitrary orthonormal basis of $\R^{N+1}.$

Do SIC sets exist? 

\begin{exa}\rm
  A SIC set in $\CP^1=S^2$ is an inscribed tetrahedron in the
  two-sphere. Any two such tetrahedrons are congruent by
  $\SO(3)=\SU(2)/\Z_2,$ though that does not stop us seeking the
  `neatest' set of vertices to write down. One set is
\begin{equation}\Big\{\ [0,1],\quad [\sqrt2,1],\quad
[\sqrt2,\om],\quad [\sqrt2,\om^2]\ \Big\},\end{equation} where
$\om=e^{2\pi i/3}.$ Another set of vertices, which is perhaps less obvious, is
\begin{equation}\label{tet2}
\Big\{\ [1,\varpi],\quad [\varpi, 1],\quad
[1,-\varpi],\quad [\varpi,-1]\ \Big\},
\end{equation} 
where $\varpi=(1+i)/(1+\sqrt3).$ This second set nevertheless plays an
important role, as we shall see.
\end{exa}

If $n\ge3,$ any two SIC sets in $\CP^{n-1}\subset\R^N$ are congruent by
$\SO(N)$ (where $N=n^2-1$), but \emph{not} in general by $\SU(n).$

One can present more SIC sets by generalizing the second tetrahedron
\eqref{tet2}. We define two cyclic groups of order $n.$ Let $W$ be the
group generated by the cyclic permutation
\begin{equation}\label{cyc} 
[z_1,z_2,\ldots,z_n]\mapsto[z_n,z_1,\ldots,z_{n-1}];
\end{equation}
let $\om=e^{2\pi i/n},$ and denote by $H$ the group generated by
\begin{equation}\label{run} 
[z_1,z_2,\ldots,z_n]\mapsto[z_1,\om z_2,\ldots,\om^{n-1}z_n].
\end{equation}
$\WH$ acts on $\CP^{n-1}$ as a subgroup of $\SU(n)$ isomorphic to
$\Z_n\times\Z_n.$ This subgroup is sometimes called the
Weyl-Heisenberg group after \cite{Weyl}. It can be regarded as the
projectivization of an extended finite group, namely the Heisenberg
group of three-by-three matrices with coefficients in the ring $\Z_n.$
For this reason, it is legitimate to refer to the action of $\WH$
simply as that of the \emph{Heisenberg group}. The following two
results can be verified by direct calculation:

\begin{prop}\label{011}
The orbit
\begin{equation} 
(\WH)\cdot[0,1,1]
\end{equation} 
is a SIC set consisting of nine points in $\CP^2.$
\end{prop} 

\begin{prop} 
Let $r=\sqrt2$ and $s=\sqrt{2+\sqrt5}.$ Then
\begin{equation}
(\WH)\cdot[-s-i(r\+s),\ 1\-r+i,\ s+i(s\-r),\ 1+r+i]
\end{equation}
is a SIC set of sixteen points in $\CP^3.$
\end{prop} 

\noindent An element $\z\in\C^n$ such that the orbit $(\WH) \cdot[\z]$
is a SIC set is called a \emph{fiducial vector} for the action of
$\WH.$\smallbreak

In his 1999 Vienna PhD thesis \cite{Z}, Zauner made a number of 
conjectures that extended the basic
 
\begin{conj}
$\CP^{n-1}$ possesses a SIC set for all $n.$
\end{conj}

\noindent It is widely believed that such a set can always be realized
as an orbit of $\WH,$ and that the number of non-congruent solutions
(meaning solutions that are not related to one another by an isometry
or element of $\SU(n)$) increases with $n.$ There are sporadic
constructions of SIC sets using different finite groups (see
Remark~\ref{HP3}).\medbreak

Explicit algebraic solutions are known for
$n=2,3,4,\ldots,15,19,24,35$ and $48,$ from work of Zauner \cite{Z},
Appleby \cite{App}, Renes \etal \cite{RSC}, Flammia \cite{Fla},
Grassl \cite{Grl}, Zhu \cite{Zhu}, and many other authors (see
\cite{AFZ,DBBA} and references cited therein). All such examples lie
(up to isometry) in solvable extensions of $\mathbb Q$
\cite{AYZ}. Extensive numerical verification has been carried out for
$n\le67$ (Scott and Grassl \cite{SG}).

The next result is well known, but we include it
for completeness. Let $\{[\z_i]\}$ be a SIC set in $\CP^{n-1}$ and
$\{P_i\}$ its image in $\sH_n.$ Recall that $\tr(P_iP_j)=\kappa$ if $i\ne
j.$ Thus $\kappa$ is the cross ratio or transition probability between
any two points in the SIC set.

\begin{lem}\label{CSI}
Any SIC set in $\CP^{n-1}$ satisfies
$\kappa=1/(n+1),$ and \begin{equation}\frac1n\sum P_i=I.\end{equation}
\end{lem} 

\begin{proof} 
Define $Q_j=P_j-\kappa I.$ Then
\begin{equation} \tr(P_iQ_j)=
\left\{\ba{ll} 1-\kappa\quad & i=j\\0& i\ne j\ea\right.\end{equation}
So $(P_i)$ is a basis of $i\mathfrak{u}(n)$ (called a \emph{quorum})
and we can set
\begin{equation}\label{I}
I=\sum_{i=1}^{n^2} c_iP_i.
\end{equation}
Applying $\tr(Q_j\,\cdot),$ we get $1-\kappa n=c_j(1-\kappa),$ so all the
$c_i$ are equal. To complete the proof, take the trace of
\eqref{I}. This gives
\begin{equation} n=n^2\frac{1-\kappa n}{1-\kappa},\end{equation}
and $\kappa=(1-n)/(1-n^2)=1/(1+n).$
\end{proof}

It will be convenient in our analysis of SIC sets to introduce the
following.

\begin{defi}\label{cs} 
Two points in $\CP^{n-1}$ will be said to be `correctly separated' if
the cross ratio that they define equals $1/(n+1).$
\end{defi}

Suppose that $\CP^{n-1}$ admits a SIC set $\SS.$ Then, up to isometry,
two points form part of a SIC set if and only if they are correctly
separated.  This follows from the fact that $\CP^{n-1}$ is a
\emph{two-point homogeneous space}, meaning that there exists an
isometry that maps any two points to any other two points the same
distance apart \cite{Wang}. The lemma above is then a key result that
enables one to go some way in attempting to construct a SIC set
without knowing for sure that it exists.

\begin{rem}\label{HP3}\rm 
Lemma~\ref{CSI} precludes the existence of four or more points of a
SIC set from lying on a projective line $\CP^1$ whenever $n\ge3,$
since their cross ratio would have to be that for $n=2,$ namely $1/3.$
An application relates to the SIC set in $\CP^7$ constructed by Hoggar
\cite{Hogg}. It consists of a $(\Z_2)^6$ orbit of 64 points that the
Hopf fibration $\pi\colon\CP^7\to\HP^3$ projects down to an equal
number of points in the quaternionic projective space $\HP^3.$ It
would be impossible to find a SIC set in $\CP^7$ with four points in
each fibre of $\pi,$ but we wonder whether there exists a SIC set
arising from 32 points in $\HP^3$ with two points in each fibre. Such
questions are related to work by Armstrong \etal on twistor lifts
\cite{APS}.
\end{rem}

\sect{The action of a maximal torus}\label{mt}

Starting in this section, we restrict the discussion mainly to the
case $n=3.$ We shall develop the concept of moment mapping, but
restricted to a maximal torus in $\SU(3),$ acting on $\CP^2.$ Fix the
torus
\begin{equation}\label{T}
T = \left\{\left(\!\ba{ccc} 
e^{ix_1} & 0 & 0\\
0 & e^{ix_2} & 0\\
0 & 0 & e^{ix_3}
\ea\!\right)\ :\ \ts
\sum\limits_{i=1}^3 x_i=0\hbox{ mod\,$2\pi$}\right\},
\end{equation}
which is, of course, homeomorphic to $S^1\times S^1.$ The hyperplane
$x_1+x_2+x_3=0$ in $\R^3$ represents the Lie algebra $\mathfrak t$ of
$T,$ which we also identify with $\mathfrak{t}^*$ using the induced
inner product. The moment mapping for $T$ acting on $\CP^2$ is then
the composition
\begin{equation}
\CP^2 \lra\su(3)\lra \mathfrak{t}
\end{equation}
obtained by projecting the adjoint orbit orthogonally to $\mathfrak t.$

When we pass from $\su(3)$ to $\sH_3$ via \eqref{obv}, we can identify
this composition with the mapping $[\z]\mapsto(x_1,x_2,x_3),$ where
\begin{equation}\label{mu}
(x_1,x_2,x_3)=\mu([\z])=\frac1{\|\z\|^2}(|z_1|^2,|z_2|^2,|z_3|^2)
\end{equation}
consists of the diagonal entries in \eqref{can}. Here,
$\|\z\|^2=|z_1|^2+|z_2|^2+|z_3|^2,$ though it is convenient to assume
$\|\z\|=1.$ After the shift from traceless matrices to $\sH_3,$ the
image of $\mu$ is the two-simplex $\sT,$ a filled equilateral triangle
lying in the plane $x_1+x_2+x_3=1,$ illustrated in Figure~1. The
residual three-fold symmetry visible is that of the Weyl group
$W=N(T)/T\cong\Z_3.$

It is well known that $\sT$ parametrizes the orbits of $T$ on $\CP^2$
via \eqref{mu}. See, for example, Guillemin and Sternberg \cite{GS}.
The inverse image of an interior point of $\sT$ is a two-torus
$T/\Z_3;$ the inverse image of a vertex is a single point in $\CP^2;$
and the inverse image of any other boundary point is a circle $S^1.$
Topologically, this leads to a description of the complex projective
plane as a quotient
\begin{equation}
 \CP^2 = \frac{\sT\times T^2}{{}\sim{}}.\end{equation}
 Here $\sim$ is the equivalence relation that collapses points over the
boundary of $\sT$ in accordance with the scheme outlined above.

\begin{figure}[b]
\scalebox{1.4}{\includegraphics{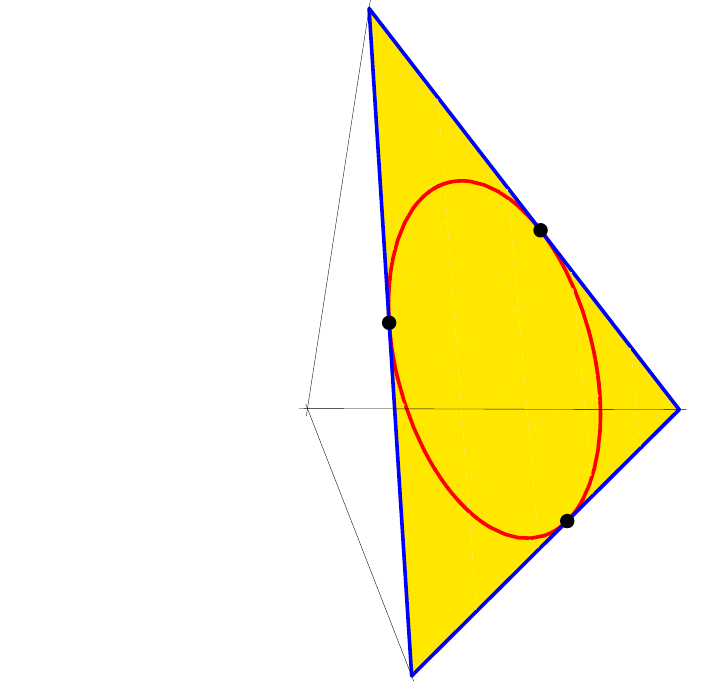}}
\caption{The image of the moment mapping $\mu\colon\CP^2\to\R^3$ is a
  two-simplex $\sT$ that takes the form of a filled equilateral
  triangle. If $\z=(z_1,z_2,z_3)$ is a unit vector in $\C^3,$ the
  point $[\z]$ is mapped to $(x_1,x_2,x_3)=(|z_1|^2,|z_2|^2,|z_3|^2).$
  The inscribed circle is the intersection of the plane
  $x_1+x_2+x_3=1$ containing the (coloured) image of $\mu$ and the
  (invisible) sphere $x_1^2+x_2^2+x_3^2=\ha.$}
\end{figure}

Let $m_1,m_2,m_3$ be the midpoints of the sides of $\sT,$ and consider
the circles
\begin{equation} \label{cimi} 
C_i=\mu^{-1}(m_i),\qquad i=1,2,3.
\end{equation} 
The first circle $C_1$ consists of those points $[0,z_2,z_3]$ of
$\CP^2$ with $|z_2|=|z_3|.$ Any set of three
equidistant points in $C_1$ has the form
\begin{equation}\label{phs}
 [0,e^{i\si},1],\quad [0,e^{i\si},\om],\quad [0,e^{i\si},\om^2],
\end{equation}
where $\om=e^{2\pi i/3}.$ The cross ratio defined by any two of these
points is given by
\begin{equation}\ts 
\left|\ha(1+\om)\right|^2 = \frac14,
\end{equation}
so they are indeed correctly separated. Similarly, $C_2$ consists of
points $[z_1,0,z_3]$ with $|z_1|=|z_3|,$ and $C_3$ points
$[z_1,z_2,0]$ with $|z_1|=|z_2|.$ Now choose three equidistant points
in $C_2,$ and three equidistant ones in $C_3.$ It is easy to check
that the resulting nine points constitutes a SIC set. This generalizes
Proposition~\ref{011}.

\begin{defi} 
By a midpoint solution, we mean a SIC set in $\CP^2$ consisting of
three points in each of the three circles $C_1,C_2,C_3.$
\end{defi}

\noindent This construction defines a one-parameter family of SIC sets
up to isometry, since the stabilizer in $\SU(3)$ of the points in
$C_1$ is a subgroup $\U(1)$ that can be used to remove the phase
ambiguity in $C_2.$ See the discussion surrounding \eqref{U1}.

Let $\sC$ denote the circle passing through the midpoints
$m_1,m_2,m_3,$ illustrated in Figure~1. As a curve in $\R^3,$ it is
the intersection of the plane $x_1+x_2+x_3=1$ with the sphere
$x_1^2+x_2^2+x_3^2=1/2.$ It was actually plotted using the next
result.

\begin{lem}\label{paraC} 
In $\R^3,$ the inscribed circle $\sC$ is parametrized by
\begin{equation}\fr23\Big(\!\cos^2\th,\,\cos^2(\th+\fr23\pi),\,
\cos^2(\th+\fr43\pi)\Big),\end{equation} for
$\th\in[-\ha\pi,\ha\pi).$
\end{lem} 

\begin{proof} 
First, consider the effect of $\mu$ on \emph{real} vectors. Suppose
that $\z=(a,b,c)$ is a unit vector with $a,b,c\in\R,$ and set
$s_k=a^k+b^k+c^k.$ There is an identity
\begin{equation}\label{Sid}
s_1(-a+b+c)(a-b+c)(a+b-c)=s_2^2-2s_4.
\end{equation}
If $\mu([\z])=(a^2,b^2,c^2)\in\sC$ then $s_2=1$ and $s_4=1/2,$ so the
right-hand side of \eqref{Sid} vanishes. It follows $a\pm b\pm c=0$
for some choice of signs. Now set
\begin{equation}\label{abc=}\ts
a=\sqrt{\frac23}\cos \th,\quad 
b=\sqrt{\frac23}\cos(\th+\fr23\pi),\quad 
c=\sqrt{\frac23}\cos(\th+\fr43\pi).
\end{equation}
Trig-expanding $b$ and $c$ shows that $a+b+c=0$ and $a^2+b^2+c^2=1.$
It follows from \eqref{Sid} that $s_4=1/2$ and $(a^2,b^2,c^2)\in\sC.$

The midpoints $m_1,m_2,m_3$ are given respectively by
$\th=\pm\pi/2,-\pi/6,\pi/6.$ This confirms the stated range for $t.$
\end{proof}

\sect{Weyl-Heisenberg orbits}

In this section, we show how the moment mapping \eqref{mu} helps one
to understand the action of the groups $W$ and $H$ defined in
\eqref{cyc} and \eqref{run} with $n=3.$ We shall see that the $\sC$
plays a prominent role, and Lemma~\ref{paraC} will be the basis
for the parametrization of elements of a SIC set.

\begin{lem}\label{nicefibres}
The $H$-orbit of a point $[\z]$ in $\CP^2$ consists of three points
that are correctly separated from one another if and only if
$\mu([\z])\in\sC.$
\end{lem}

\begin{proof}
Suppose that $\z=\z^{(0)}$ is a unit vector. The orbit $H\cdot[\z]$
consists of the projective classes of the vectors
\begin{equation}\label{Horb}
\z^{(0)}=(z_1,z_2,z_3),\quad 
\z^{(1)}=(z_1,\om z_2,\om^2 z_3),\quad
\z^{(2)}=(z_1,\om^2 z_2,\om z_3)
\end{equation}
generated by \eqref{run}. We can express
\begin{equation}
|\big<\z^{(0)},\z^{(1)}\big>|^2
=(|z_1|^2+ \om|z_2|^2+\om^2|z_3|^2)(|z_1|^2+\om^2|z_2|^2+\om|z_3^2|)
\end{equation}
in the form $\alp-\beta,$ where
\begin{equation} 
\alp=|z_1|^4+|z_2|^4+|z_3|^4,\quad 
\beta=|z_2|^2|z_3|^2+|z_3|^2|z_1|^2+|z_1|^2|z_2|^2.
\end{equation}
Therefore $\z^{(0)}$ and $\z^{(1)}$ are correctly separated if and only if
$\alp-\beta=1/4.$ But 
\begin{equation}\label{JK}
\alp+2\beta=(|z_1|^2+|z_2|^2+|z_3|^2)^2=1,
\end{equation}
since $\z=\z^{(0)}$ is normalized, so the condition of correct
separation is $\alp=1/2.$ Since $x_i=|z_i|^2$ are the Cartesian
coordinates in $\R^3,$ correct separation of $\z^{(0)}$ and $\z^{(1)}$
implies that $\mu([\z])\in\sC.$ This condition only depends on
$\mu([\z])$ since $H$ is a subgroup of $T$ and its action commutes
with all the elements of $T.$ Therefore if $\mu([\z])\in\sC,$ all
three points in \eqref{Horb} will be correctly separated.
\end{proof}

\begin{exa}\rm
Lemma~\ref{nicefibres} is really an assertion about the induced metric
on the fibres $\mu^{-1}(p)$ for $p\in\sC.$ This metric will depend
crucially on the position of $p$ in $\sC,$ since it degenerates as $p$
approaches any one of the midpoints $m_i$ (over which the fibres are
circles rather than 2-tori). This behaviour is illustrated in
Figure~2, which provides a visualization of the fibres $\mu^{-1}(p_i)$
for $i=1,2,$ where
\begin{equation}\label{p12}
p_1=\big(\fr23,\fr16,\fr16\big),\qquad
p_2=\big(\fr18(3+\sqrt5),\fr14,\fr18(3-\sqrt5)\big).
\end{equation}
are two points of $\sC$. Note that $p_1$ is the point
  diametrically opposite $m_1,$ whereas $p_2$ lies between $p_1$ and
  $m_3.$

The coordinates used in Figure~2 are derived from the action of the
maximal torus \eqref{T}, which is represented by translation. Scalar
multiplication by $\om=e^{2\pi i/3}$ on vectors in $\C^3$ generates
the action of the centre $\Z_3$ of $\SU(3),$ so that $(z_1,z_2,z_3)$
and $(\om z_1,\om z_2,\om z_3)$ appear as distinct points in the
diagrams, although they determine the same point of $\CP^2.$ The
centre is responsible for the evident three-fold symmetry, which is
best represented by the hexagonal fundamental domain on the right-hand
side. Comparing this with the left-hand parallogram and its
translates, one sees that a 2-torus can be formed by identifying the
opposite edges of a hexagon, a fact that is well known (see, for
example, Thurston \cite{Thur}).

\begin{figure}[t]
\scalebox{.95}{\includegraphics{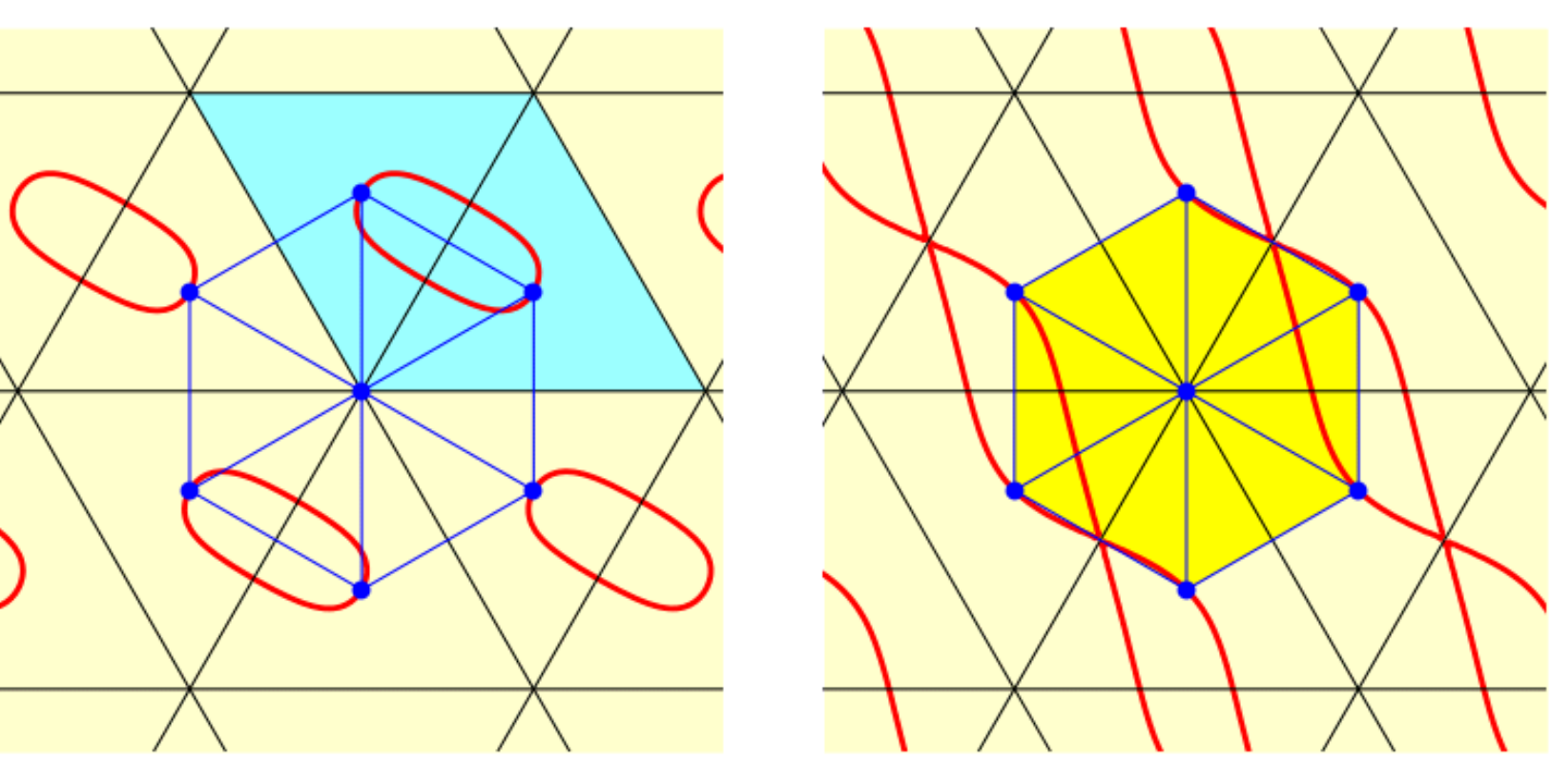}}
\caption{A representation of the fibre $\mu^{-1}(p_1)$ (left) and
  $\mu^{-1}(p_2)$ (right) in $\CP^2$ for the points \eqref{p12}. The
  coloured regions are two different fundamental regions for the
  torus, and the red curves are points a distance $2\pi/3$ from the
  centre point.}  \vskip10pt
\end{figure}

Both diagrams display exactly three distinct points of $\CP^2$ in the
closure of each coloured fundamental domain, and each of these triples
of points forms an equilateral triangle. This can be seen from an
inspection of the curves that are the loci of points a distance
$2\pi/3$ from the centre point. The latter is correctly separated from
each of the other two points, and these two points are correctly
separated from each other because distances are translation
invariant.
\end{exa}

We are now in a position to give a full description of those SIC sets
that are orbits of the group generated by \eqref{cyc} and \eqref{run}.

\begin{theorem}\label{class}
Let $\z=(z_1,z_2,z_3)\in\C^3.$ Then $(\WH)\cdot[\z]$ is a SIC set if
and only if one of the variables $z_1,z_2,z_3$ vanishes, or
\begin{equation}\label{cos3}
[\z] = 
\Big[\!\cos\th,\,\om^j\cos(\th+\fr23\pi),\,\om^k\cos(\th+\fr43\pi)\Big],
\end{equation}
for some $\th\in[-\ha\pi,\ha\pi)$ and $j,k\in\{0,1,2\}.$ 
\end{theorem}

\begin{proof}
Suppose that $\z$ is a unit vector and that $(\WH)\cdot[\z]$ is a SIC
set.

Let $\z'=(z_3,z_1,z_2).$ In the notation \eqref{Horb}, we have
\begin{equation}\label{RDe}
|\big<\z^{(k)},\z'\big>|^2 
= \big|z_1\oz_3+\om^kz_2\oz_1+\om^{2k}z_3\oz_2\big|^2
= \beta + 2\,\mathrm{Re}\big[\om^{2k}\Delta\big],
\end{equation} 
where
\begin{equation}\label{De}
\Delta=z_1^2\oz_2\oz_3+z_2^2\oz_3\oz_1+z_3^2\oz_1\oz_2.
\end{equation}
By assumption, \eqref{RDe} equals $1/4$ for all $k=0,1,2.$ From
\eqref{JK} we have $\beta=1/4,$ so the expression in square brackets
above must be purely imaginary. This happens for all $k$ if and only
if $\Delta=0.$

By assumption, $\mu([\z])\in\sC,$ so $[\z]$ must lie in a $T$-orbit of
$[a,b,c]$ where $a,b,c$ are given by \eqref{abc=} for some $\th.$
Since $a+b+c=0,$ \eqref{De} and \eqref{RDe} tell us that $\z=(a,b,c)$
is a fiducial vector. Let us look for other fiducials in the same
$T$-orbit by considering
\begin{equation}\label{zzz}
(z_1,z_2,z_3)=(a,\,e^{i\beta}b,\,e^{i\gamma}c),
\end{equation}
having normalized the coefficient of $a.$ Let us assume that
$abc\ne0,$ so $b+c\ne0.$ Since $\Delta=0,$ we have
\begin{equation}
e^{3i\beta}b+e^{3i\gamma}c = b+c.
\end{equation}
Taking the moduli of both sides gives $\cos(3\beta-3\gamma)=1,$ so
$\beta$ equals $\gamma$ mod $2\pi/3.$ It follows that both $\beta$
and $\gamma$ are multiples of $2\pi/3,$ and that $[\z]$ has the
form \eqref{cos3}.

Conversely, the vector \eqref{cos3} satisfies \eqref{RDe} and projects
to $\sC.$ Thus its $\WH$ orbit is a SIC set.
\end{proof}

To summarize, any three equally-spaced points on $\sC$ form the `base'
of a group covariant SIC set. If these are the three midpoints $m_i$
of the sides then \emph{any} point in $\mu^{-1}(m_i)$ is a fiducial
vector. But for a generic point $p\in\sC,$ the choices are restricted
to nine points on the two-torus $\mu^{-1}(p).$ As $p$ approaches a
midpoint, these nine points become three.

\begin{rem}\rm 
The methods of this section can be extended to the study of SIC sets
in $\CP^{n-1}$ that arise as orbits of $\WH$ for $n>3.$ Using the
moment mapping $\mu\colon\CP^{n-1}\to\R^n,$ one can define a subset of
the simplex $\mu(\CP^{n-1})$ consisting of points whose inverse image
contains $H$-orbits of correctly-separated points. For $\CP^3$ the
relevant subset consists of two circular arcs inside a solid
tetrahedron, but is no longer one-dimensional if $n>4,$ as discussed
by Lora Lamia \cite{NLLD}. For applications of the use of
$\mu\colon\CP^3\to\R^4$ in classifying almost-Hermitian structures on
manifolds of real dimension six, see Mihaylov \cite{Mih}.
\end{rem} 

The SIC sets in $\CP^2$ described above have been discussed by Renes
\etal \cite{RSC}, Zhu \cite{Zhu}, and various other authors. In
particular, it is known that any SIC set arising from
Theorem~\ref{class} is isometric to a midpoint solution. This can be
proved by adapting the proof of Proposition~\ref{7} below, but we
shall prove a much stronger result in this paper, namely

\begin{theorem}\label{strong}
Any SIC set $\SS$ in $\CP^2$ is congruent modulo $\SU(3)$ to a
midpoint solution.
\end{theorem} 

\noindent In the next section, we shall work with yet another
description of the isometry class of a midpoint solution, in which
each circle $C_i$ contains exactly two points of the SIC set.

\sect{Two-point homogeneity}\label{2pH}

Suppose, going forward, that $\SS$ is a SIC set in $\CP^2,$ consisting
of nine points $[\z_i],$ $i=1,\ldots,9.$ Up to the action of the
isometry group, we are free to assume that $\SS$ contains the two
points of $C_1$ represented by the unit vectors
\begin{equation}\label{zz12}
\ts\z_1=\fr1{\sqrt2}(0,1,-\om),\quad \z_2=\fr1{\sqrt2}(0,1,-\om^2),
\end{equation} 
which are a distance $2\pi/3$ apart. This is on account of the
two-point homogeneity of $\CP^2.$ Lemma~\ref{CSI} tells us that any
other point $[\z]$ of $\SS$ must satisfy
\begin{equation}\label{must}
\big|\big<\z,\z_j\big>\big|^2 = \fr14\|\z\|^2,\qquad j=1,2.
\end{equation}
Using this equation, we can prove another lemma that emphasizes the
important role played by the incircle $\sC.$

\begin{lem}\label{Zst} 
The moment map $\mu$ projects any remaining point $[\z]$ of $\SS$ to a
point of $\sC.$ Indeed, we may take $\z$ to be a unit vector of the form
\begin{equation}\label{zst}
\z(\si,\phi) = \sqrt{\!\fr23}\Big(\,e^{i\si}\cos \phi,\ 
\cos(\phi+\fr23\pi),\,\cos(\phi+\fr43\pi)\,\Big)
\end{equation}
for some $\si\in(-\pi,\pi]$ and some $\phi\in(-\ha\pi,\ha\pi].$
\end{lem} 

\noindent To lighten the notation, we shall write $\z[\si,\phi]$ as a
shorthand for $[\z(\si,\phi)],$ so that square brackets on either side
of `$\z$' indicate a projective class.\smallbreak

Lemma~\ref{paraC} tells us that $\z[\si,\phi]$ lies over $\sC.$
Observe that $\mu(\z[\si,\phi])$ depends only on the \emph{angle}
$\phi$ measured around $\sC,$ and not on the \emph{phase} $\si.$
Moreover, as $\si$ and $\phi$ vary, $\z[\si,\phi]$ parametrizes a
pinched two-torus, the pinch point being
\begin{equation}\ts
\z[\si,-\frac12\pi]=\z[\si,\frac12\pi]=[0,1,-1],
\end{equation}
which is evidently independent of $\si.$ Having chosen $[\z_1],[\z_2]$
on $C_1,$ we can see that any third point of $C_1\cap\SS$ must be this
point, which explains the pinching. One should note that
$\z(\si,\phi)=-\z(\si,\phi+\pi),$ which is why $\phi=-\pi/2$ is
excluded from the non-projective representation \eqref{zst}.\medbreak

\begin{proof}[Proof of Lemma~\ref{Zst}] 
Let us suppose that $[\z] = [z_1,z_2,z_3]\in\SS.$ If $z_2=0$ then
$[\z]\in C_2$ and \eqref{zst} will be valid for $\phi=-\pi/6.$ We may
therefore take $z_2=1$ and set $z_1=a,$ $z_3=c$ where $a,c\in\C.$ Then
by assumption, we have
\begin{equation} 
|1-c\ol\om|^2 = |1-c\om|^2,
\end{equation}
which implies that $c$ is real, and
\begin{equation} 
[\z]=\left[a,\ 1,\ \pm|c|\right].
\end{equation}
Using \eqref{must}, we see that
\begin{equation} 
\fr12(1\pm|c|+|c|^2)=\fr14(|a|^2+1+|c|^2),
\end{equation}
so $|a|^2= (1\pm|c|)^2,$ and
\begin{equation} 
|a|^2+1+|c|^2= 2(1 \pm|c|+|c|^2).
\end{equation}
Therefore,
\begin{equation} \ba{rcl} 
|a|^4+1+|c|^4 
&=& 2\pm4|c|+6|c|^2\pm4|c|^3+2|c|^4\y 
&=& \fr12(|a|^2+1+|c|^2)^2.
  \ea\end{equation} 

It follows that $\mu$ does indeed map $[\z]$ into $\sC.$ In view of
Lemma~\ref{paraC}, we must be able to express $[\z]$ in the stated
form for some $e^{i\si}\in\U(1).$
\end{proof}

The points $[\z_1],[\z_2]$ are both fixed by the subgroup $\U(1)$ of \eqref{T}
generated by
\begin{equation}\label{U1}
\left(\!\ba{ccc} 
e^{-2ix} & 0 & 0\\
0 & e^{ix} & 0\\
0 & 0 & e^{ix}
\ea\!\right),
\end{equation}
so we may assume that a third point of $\SS$ is $\z[0,\th]$. The next
result shows that there \emph{does} exist a SIC set containing this
point for any $\th.$

\begin{prop}\label{7} 
For any $\th\in[-\ha\pi,\ha\pi),$ the six points 
\begin{equation}\label{6pts}
\ba{lll}
[0,1,-\om]=[\z_1],\qquad & [0,1,-\om^2]=[\z_2] &\in C_1,\\[5pt]
[1,0,-\om]=\z[-\frac23\pi,\!-\sa\pi],\quad &
[1,0,-\om^2]=\z[\frac23\pi,\!-\sa\pi] &\in C_2,\\[5pt]
[1,-\om,0]=\z[-\frac23\pi,\sa\pi], &
[1,-\om^2,0]=\z[\frac23\pi,\sa\pi] &\in C_3,
\ea\end{equation}
combine with three points
\begin{equation}\label{3pts}\ts
\z[0,\th-\frac13\pi],\qquad \z[0,\th],
\qquad\z[0,\th+\frac13\pi]
\end{equation}
to form a SIC set isometric to a midpoint solution.
\end{prop} 

\noindent We shall denote this SIC set by $\SSt;$ it is illustrated in
Figure~3.

\begin{proof} 
Consider the matrix
\begin{equation} 
M = \frac1{\sqrt3}\!\left(\ba{ccc} \om^2 & \om & 1 \\ 1 & \om &
\om^2\\ 1 & 1 & 1 \ea\right).
\end{equation} 
It is easy to check that $M\in\U(3)$ and that $M^3=i\om^2I.$ A
calculation shows that
\begin{equation}\ts 
M\kern1pt\z(0,\th)^{\!\top} = 
\fr1{\sqrt2}\!\left(e^{i\th}\om^2,\ e^{-i\th},\ 0\right)^{\!\top},
\end{equation}
so that $M$ maps $\z[0,\th]$ to the point $[e^{2i\th},\,\om,\,0]$
of $C_3.$ Moreover, $M$ maps the array \eqref{6pts} to the array
\begin{equation}\ba{ll}
[0,1,-1], & [1,0,-\om],\\[5pt] 
[1,0,-\om^2],\qquad &[0,1,-\om^2],\\[5pt] 
[0,1,-\om], & [1,0,-1] \ea\end{equation} of points
in $C_1\sqcup\,C_2.$ It follows that $M$ maps $\SSt$ onto three
triples of points, each triple belonging to $C_i$ for some $i=1,2,3.$
\end{proof}

The first six points \eqref{6pts} of $\SSt$ do not depend on $\th,$
whereas the last triple of points can be rotated at will (by varying
$\th$) around a circle $C_3'$ covering $\sC$. For example, $\z[0,0]$
lies over the point $p_1$ of $\sC$ diametrically opposite $m_1$ (see
\eqref{p12}).

\begin{figure}[b]
\scalebox{.72}{\includegraphics{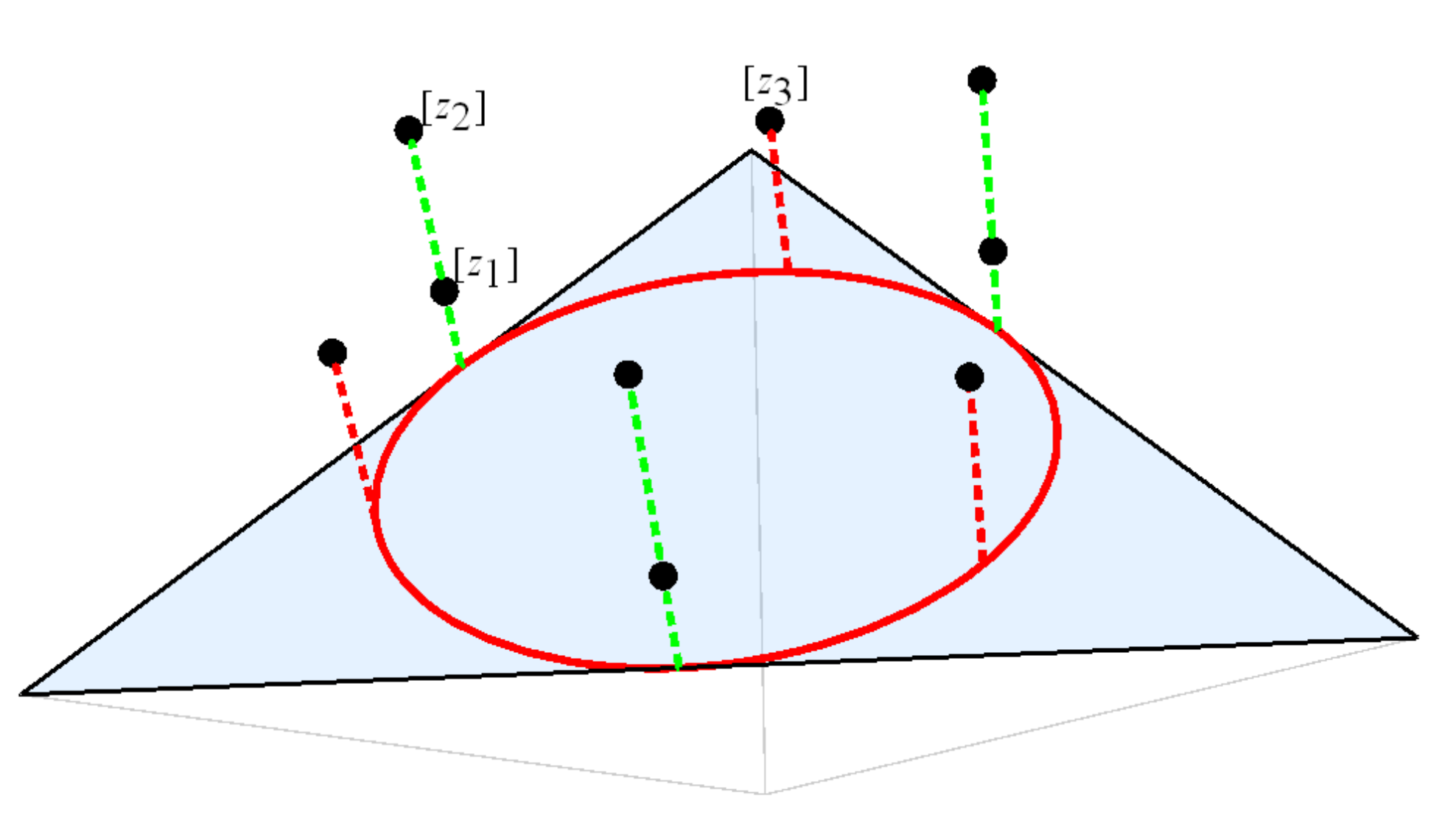}}
\caption{The SIC set $\SSt$ defined by Proposition~\ref{7} contains
  two points in each circle $C_i$ (including $[\z_1],[\z_2]$ in $C_1$)
  that do not depend on $\th,$ together with a triple of points
  (including $[\z_3]$) that $\mu$ also projects to $\sC$ for which
  $\th$ represents the angle around $\sC.$} 
\end{figure}

\begin{rem}\rm 
Nine points in $\CP^2$ are the inflection points of a non-singular
cubic curve if and only if the line determined by any two of them
contains a third. This being the case, there are twelve such lines
altogether, on which the nine points lie by threes, with four of the
twelve lines through each of the nine points, thus forming the
so-called Hesse configuration $\{9_4, 12_3\}$. For the points of
$\SSt$ to arise in this way, and as described by Hughston \cite{Hugh}
and Dang \etal \cite{DBBA}, the projective line $\LL_1\cong\CP^1$
generated $[\z_1],[\z_2]$ must contain a third point of $\SSt.$ But
$\LL_1$ is the inverse image by $\mu$ of the side of $\sT$ containing
$m_1,$ and will only contain another point if $\th$ assumes one of the
values $\pm\pi/2,\pm\pi/6.$ This occurs when the three red legs (the
ones generated by $[\z_3]$ by rotation by $2\pi/3$) in Figure~3 line
up with the green legs (the ones over the midpoints), and $\SSt$ is
then itself a special midpoint solution.
\end{rem} 

\begin{exa}\label{eigen}\rm 
The unitary transformation $M$ maps $C_3'$ to $C_3.$ It permutes the
elements of the SIC set $(\WH)\cdot[0,1,-1],$ though it fixes none of
them. The matrices
\begin{equation}
A = \left(\!\ba{ccc}0&1&0\\0&0&1\\1&0&0\ea\!\right),\quad
B = \left(\!\ba{ccc}1&0&0\\0&\om&0\\0&0&\om^2\ea\!\right)
\end{equation}
generate $W$ and $H$ respectively, and satisfy 
\begin{equation}\label{conj}
M\kern-1pt AM^{-1} = \om B,\qquad
M\kern-1pt BM^{-1} = \om^2A^{-1}B^{-1}.
\end{equation}
It follows that $M$ is an element of the so-called Clifford group, the
normalizer of $\WH$ in $\U(3).$ Modulo phase, this normalizer is
isomorphic to a semidirect product
$\mathrm{SL}(2,\Z_3)\ltimes(\Z_3)^2$ (see Appleby \cite{App} and
Horrocks-Mumford \cite{HM}). Equation \eqref{conj} asserts that $M$
induces the automorphism of $\WH$ given by \begin{equation}
  \left(\!\ba{cc}0&-1\\1&-1\ea\!\right)\in\mathrm{SL}(2,\Z_3).
\end{equation}
It is conjectured that a fiducial vector can always be found in an
eigenspace of some element of the Clifford group (see Zauner
\cite{Z}). In the case of $M,$ a computation shows that any one of its
eigenvectors defines a point of $\CP^2$ whose orbit under $\WH$ is a
configuration of nine points arranged in nine lines. Each of the 27
pairs of points lying on one of the nine lines has a cross ratio
$\kappa=1/3,$ whereas the remaining nine pairs of points have
$\kappa=0.$ Compared to the Hesse configuration above, this means that
three of the twelve triples of points are not collinear, but each of
these three triples forms an orthonormal basis of $\C^3.$
\end{exa}

\begin{rem}\rm 
If an isometry is to fix both $[\z_1]$ and $[\z_2],$ there is no
ambiguity remaining in the choice of $\phi\in[\-\pi/2,\pi/2)$ in
  Lemma~\ref{Zst}. However, we are at liberty to interchange $[\z_1]$
  and $[\z_2]$ by applying either complex conjugation or the unitary
\begin{equation}
\left(\!\ba{ccc} 1 & 0 &
  0\\0 & 0 & 1\\0 & 1 & 0\ea\!\right).
\end{equation} 
The former has the effect of replacing $\si$ by $-\si,$ and the latter
of replacing $\phi$ by $-\phi$ in \eqref{zst}. In particular, the
congruence class of the unordered set $\SSt$ uniquely specifies
$|\th|.$ This fact can also be verified using a triple product that
measures the signed area of the planar geodesic triangle spanned by
three points. See, for example, Brody and Hughston \cite{BH} and
references cited therein.
\end{rem}

\sect{Trigonometry}

From now on, we shall assume that $\SS$ is a SIC set in $\CP^2$ that
contains the points $[\z_1],[\z_2]$ defined by
\eqref{zz12}. Lemma~\ref{Zst} tells that any other point of $\SS$ has
the form
\begin{equation}\label{ZST} 
\z[\si,\phi]=\Big[\,e^{i\si}\cos\phi,\
\cos(\phi+\fr23\pi),\,\cos(\phi+\fr43\pi)\om^2\,\Big],
\end{equation} 
where $(\si,\phi)$ belongs to the rectangle
\begin{equation}
\sR = (-\pi,\pi]\times(-\fr12\pi,\fr12\pi].
\end{equation} 
The next result, from which many others follow, translates distance
into the new `rectangular' coordinates.

\begin{lem}\label{corr} 
Suppose that
$\phi,\psi\in(-\ha\pi,\ha\pi)\setminus\{-\sa\pi,\sa\pi\}.$ Then the
points $\z[\si,\phi]$ and $\z[\tau,\psi]$ are the correct distance
$2\pi/3$ apart if and only if
\begin{equation}\label{CS}
  1-\cos(\si-\tau) = 
\frac{9(1+2\cos(2(\phi-\psi)))\sec\phi\sec\psi}
{16(\cos\phi\cos\psi + 3\sin\phi\sin\psi)}.
\end{equation}
\end{lem} 

\begin{proof} 
Not only do we have to establish the formula, but we also need to show
that the assumptions imply that the denominator of the fraction is
non-zero. We use the abbreviated notation
\begin{equation}\ba{rcl}
\Ga_0 &=& \cos\phi\cos \psi,\y \Ga_1 &=&
\cos(\phi+\frac23\pi)\cos(\psi+\frac23\pi),\y \Ga_2 &=&
\cos(\phi+\frac43\pi)\cos(\psi+\frac43\pi).
\ea\end{equation} 
The condition on the cross ratio for correct separation is that
\begin{equation} 
\fr49\big|e^{i(\si-\tau)}\Ga_0+\Ga_1+\Ga_2\big|^2=\fr14,
\end{equation}
which gives
\begin{equation} 
2\cos(\si-\tau)\Ga_0(\Ga_1+\Ga_2)+\Ga_0^2+(\Ga_1+\Ga_2)^2=\fr9{16},
\end{equation}
and therefore
\begin{equation}\label{minus9}
32\Ga_0(\Ga_1+\Ga_2)(1-\cos(\si-\tau)) = 
32\Ga_0(\Ga_1+\Ga_2)+ 16\Ga_0^2+16(\Ga_1+\Ga_2)^2 - 9.
\end{equation}
A calculation shows that the right-hand side of \eqref{minus9} is equal to
\begin{equation}
9(1+2\cos(2(\phi-\psi))),
\end{equation}
which vanishes when $\cos(\phi -\psi)=\pm1/2.$ By hypothesis,
$\Ga_0\ne0.$ If
\begin{equation}
\Ga_1+\Ga_2 = \fr12[\cos \phi\cos \psi + 3\sin\phi\sin\psi]
\end{equation}
vanishes, then
\begin{equation} 
\cos \phi\cos \psi +\sin \phi\sin \psi=\cos(\phi-\psi)
=\pm\fr12,
\end{equation}
and hence
\begin{equation} 
\cos \phi\cos \psi=\pm\fr34,\quad\sin \phi\sin \psi=\mp\fr14.
\end{equation}
Now set $x=\tan \phi$ and $y=\tan\psi.$ Then $xy=-1/3$ and it holds
that
\begin{equation} 
\pm\sqrt3=\tan(\phi-\psi)=\frac{x-y}{1+xy}=\fr32(x-y).
\end{equation}
We therefore have
\begin{equation} (x+y)^2=(x-y)^2 +4xy=0,
\end{equation}
and $\phi=-\psi=\pm\pi/6,$ values that are excluded. We may
therefore assume that $\Ga_0(\Ga_1+\Ga_2)\ne0,$ and \eqref{CS}
follows.
\end{proof}

\begin{lem}\label{pi2}
If $\SS$ contains the pinch point $[0,1,-1]$ as well as $[0,1,-\om]$
and $[0,1,-\om^2],$ then $\SS$ is a midpoint solution.
\end{lem}

\begin{proof}
By hypothesis, $\SS$ contains three points of the circle $C_1.$ If
$\z[\si,\phi]$ is a fourth point of $\SS,$ then \eqref{ZST} is
correctly separated from $[0,1,-1]$ and
\begin{equation}
\cos(\phi+\fr23\pi)-\cos(\phi+\fr43\pi)=\pm\fr{\sqrt3}2.
\end{equation}
This implies that $\sin\phi=\pm1/2,$ and forces $\z[\si,\phi]$ to lie
on $C_2\sqcup C_3.$ Therefore $\SS$ lies in the disjoint union
$C_1\sqcup C_2\sqcup C_3.$
\end{proof}

One can rewrite \eqref{CS} as
\begin{equation}\label{rew}\ba{rcl}
 \cos(\si - \tau) 
&=& \ds
1-\frac{9(1 + 2 \cos 2(\phi-\psi))}
{4(4\cos^2\phi\cos^2\psi+3\sin 2\phi\sin2\psi)}\\[15pt]
&=& \ds
\frac{-5 + 4\cos2\phi+ 4\cos2\psi- 14\cos2\phi\cos2\psi-6\sin2\phi\sin2\psi}
{4(1 + \cos2\phi+\cos2\psi+\cos2\phi\cos2\psi+ 3\sin2\phi\sin2\psi)}.
\ea\end{equation}
We shall convert the right-hand side into a rational function by setting
\begin{equation}\label{tans}
x=\tan\phi,\quad y=\tan\psi.  
\end{equation} 
In the light of Lemma~\ref{pi2}, we assume from now on that $x,y$ are
finite.

Equation~\eqref{rew} simplifies to
\begin{equation}\label{stxy}
 \cos(\si - \tau) =
\frac{-11+9x^2+9y^2-27x^2y^2-24xy}{16(1+3xy)}.
\end{equation}
If $1+3xy=0,$ then the numerator on top of it must also vanish, so
$x^2+y^2=2/3$ and $(x+y)^2=0.$ Thus (as in the previous proof)
$x=-y=\pm\rec.$ This means that $\z[\si,\phi]$ lies on one of the circles
$C_2,C_3,$ and $\z[\tau,\psi]$ lies on the other, so there are no
restrictions on $\si$ and $\tau.$

The main result of this section is the following, which establishes a
criterion for the existence in $\CP^2$ of five points that are
correctly separated from one another.

\begin{theorem}\label{fa}
Suppose that $\z[\si,\phi],\z[\tau,\psi],\z[\upsilon,\chi]$ are three
points of $\CP^2,$ a distance $2\pi/3$ away from each other (and from
$[\z_1],[\z_2]$). Set $p=x+y+z,$ $q=yz+zx+xy,$ $r=xyz,$ where
$x=\tan\phi,$ $y=\tan\psi,$ $z=\tan\chi.$ Then $F(p,q,r)=0$ where
\begin{equation}\label{F0}\ba{c}
F(p,q,r)= 9 - 22 p^2 + 9 p^4 + 87 q - 126 p^2 q + 27 p^4 q + 298 q^2 - 226 p^2q^2\\ 
+ 24 p^4 q^2+ 414 q^3 - 138 p^2 q^3 + 189 q^4 + 27 q^5 - 3 p r - 50 p^3 r - 15 p^5 r\\ 
+ 88 p q r- 48 p^3 q r + 234 p q^2 r + 18 p^3 q^2 r - 144 p q^3 r + 81 p q^4 r + 189 r^2\\
 -  480 p^2 r^2 - 153 p^4 r^2 + 1398 q r^2 - 306 p^2 q r^2 +  2736 q^2 r^2- 486 p^2 q^2 r^2\\
 + 810 q^3 r^2 + 243 q^4 r^2 -  558 p r^3 - 486 p^3 r^3 + 2376 p q r^3 - 810 p q^2 r^3\\
 + 567 r^4 -  162 p^2 r^4 + 6399 q r^4 + 486 q^2 r^4 + 1701 p r^5 + 2187 r^6.
\ea\end{equation}
\end{theorem}

\begin{proof}
Although \eqref{F0} is rather complicated, the existence of such an
expression is a consequence of the elementary trigonometric identity
\begin{equation}\label{trig}
  A+B+C=0\ \Rightarrow\ \cos^2\!A+\cos^2\!B+\cos^2\!C=1+2\cos A\cos
  B\cos C,
\end{equation}
which is tailor made for \eqref{stxy}. The identity itself can be
proved by writing applying more standard ones to the sum $A+(B+C).$
Denote the right-hand side of \eqref{stxy} by the symmetric function
$c(x,y).$ Then
\begin{equation}\label{ccc}
c(x,y)^2+c(y,z)^2+c(z,x)^2 = 1 + 2c(x,y)c(y,z)c(z,x).
\end{equation}
This simplifes into the vanishing of the quotient
\begin{equation}\label{243}
\frac{243\,f(x,y,z)}{2048(1+3xy)^2(1+3yz)^2(1+3zx)^2},
\end{equation}
in which $f$ is a totally symmetric polynomial. We can then use the
\emph{Mathematica} command $\mathtt{SymmetricReduction}$ to express
\begin{equation}
f(x,y,z)=F(p,q,r)
\end{equation}
as a function of the elementary symmetric polynomials, and the result
follows.
\end{proof}

\sect{Graphical interpretation}\label{Gi}

Suppose once again that $\SS$ is a SIC set in $\CP^2$ containing
$[\z_1]=[0,1,-\om]$ and $[\z_2]=[0,1,-\om^2],$ and (in view of
Lemma~\ref{pi2}) \emph{not} the third point $[0,1,-1]$ of $C_1.$ The
planar parametrization \eqref{CS} of the remaining points of $\SS$
enables us to describe graphically the quest for such SIC sets. Before
we do this, we prove two results that help with their classification.

Setting $\phi=-\pi/6$ in \eqref{zst} defines the circle $C_2,$ and
$\phi=\pi/6$ the circle $C_3.$ It will be convenient to consider three
more circles $C_-,C_0,C_+$ given by $\si=-2\pi/3$, $0$, $2\pi/3$
respectively. Unlike $C_1,C_2,C_3,$ these three are not disjoint: they
meet in $[0,1,-1].$ The circles $C_2,C_3$ are represented by
horizontal lines in $\sR,$ and $C_-,C_0,C_+$ by equally-spaced
vertical lines; all five have diameter $\pi.$ The lines representing
$C_-,C_0$ and $C_2$ are visible in Figure~7.

\begin{lem}\label{pi6}
  If $\SS$ contains a point $\z[\si,\phi]$ with $|\phi|=\pi/6$ then
  $\SS$ is isometric to a midpoint solution.
\end{lem}

\begin{proof}
We can use the isometry \eqref{U1} to shift all points of $\SS$ by a
translation parallel to the horizontal axis within our rectangle
$\sR.$ We may therefore assume that $\si=0.$ Suppose for definiteness
that $\phi=\pi/6,$ so that $x=\rec$ and $\z[\si,\phi]\in C_3.$ Suppose
that $\z[\tau,\psi]$ is a fourth point of $\SS,$ and apply
\eqref{stxy}. The numerator equals
\begin{equation} 
-11+9x^2+9y^2-27x^2y^2-24xy=-8(1+\sqrt3y),
\end{equation} 
and the right-hand side of \eqref{stxy} becomes $-1/2$ unless
$y=-\rec.$ It follows that $\tau=\pm2\pi/3$ or $\psi=-\pi/6.$ Indeed,
the set of points correctly separated from $[\z_1],$ $[\z_2]$ and
$\z[0,\pi/6]$ is the union $C_-\cup C_2\cup C_+.$ This union must now
contain six points of $\SS,$ and no circle can contain more than
three.

Now suppose that $\SS$ contains distinct points $\z[\frac23\pi,\psi]
\in C_+$ and $\z[\upsilon,-\sa\pi]\in C_2$. Then \eqref{stxy} tells us
that either $y=\rec$ (and so $\psi=\pi/6$), or else
\begin{equation}\ts
\cos(\frac23\pi-\upsilon)=-\frac12
\end{equation}
and $\upsilon=-2\pi/3$ or $\upsilon=0$. So either the first point lies
on $C_3$, or else the second point lies on $C_-\sqcup C_0$. Now
suppose that $\SS$ contains $\z[-\frac23\pi,\psi]\in C_-$ and
$\z[\frac23\pi,\chi]\in C_+$ This time, \eqref{stxy} yields
\begin{equation}
(3y^2-1)(3z^2-1)=0,
\end{equation}
and at least one of the two points is one of the last four in
\eqref{6pts}. We may also suppose that $[0,1,-1]\not\in\SS$ by
Lemma~\ref{pi2}. It then follows that $\SS$ consists of $[\z_1]$ and
$[\z_2]$, the two points in $C_-\cap(C_2\sqcup C_3),$ the two points
in $C_0\cap(C_2\sqcup C_3)$ and three points in $C_+,$ \emph{or} the
same thing with $C_-$ and $C_+$ interchanged. Applying \eqref{U1} with
$x=\pm2\pi/3,$ we obtain exactly the SIC set $\SSt$ for some $\th$
(like the one that includes the green points in Figures~6 and 7). Then
the result follows from Proposition~\ref{7}.
\end{proof}

\begin{lem}\label{LR} 
  Suppose that $\SS$ is a SIC set that contains $[\z_1],[\z_2]$ and
  $\z[0,\th].$ Recall that any SIC set has this property up to
  isometry. If $\SS$ contains distinct points
  $\z[\si_1,\phi],\z[\si_2,\phi]$ with $\si_1,\si_2\in(-\pi,\pi)$ then
  it is isometric to a midpoint solution.
\end{lem}

\begin{proof} 
First observe that $\si_1+\si_2=0;$ this follows by applying
Lemma~\ref{corr} in which we can set $(\tau,\psi)=(0,\th)$ to obtain
$\cos\si_1=\cos\si_2.$ So take $\si=\si_1.$

In view of Lemma~\ref{pi6}, we may suppose that $x=\tan\phi$ is
different from $\pm\rec.$ We can choose a sixth point $\z[\tau,\psi]$
of $\SS$ such that $\tau\ne\pi,$ since the circle $\tau=\pi$ can
contain at most three points a distance $2\pi/3$ apart. It follows
from \eqref{stxy} that either $1+3xy=0$ (and we can apply
Lemma~\ref{pi6}) or
\begin{equation}
\cos(\si-\tau)=\cos(-\si-\tau).
\end{equation}
Since $\si=0$ and $\si=\pi$ do not yield distinct points, the only
possibility remaining from our assumption is that $\tau=0.$ If
$t=\tan\th$ and $y=\tan\psi$, \eqref{stxy} implies that
\begin{equation}
t^2+y^2-3t^2y^2-8ty-3=0.
\end{equation}
This gives 
\begin{equation}
   y=\frac{t\pm\sqrt3}{1\mp\sqrt3\,t},
\end{equation}
$\psi=\th\pm\pi/3$ modulo $\pi.$ This is the configuration of three
points visible on the central vertical axis in Figure~7. All together,
$\SS$ now contains at most seven points including $[\z_1]$ and
$[\z_2],$ which is a contradiction. Using \eqref{stxy}, one can in
fact show that given the sixth point, either $\phi$ or $\psi$ must
equal $\pm\pi/6.$
\end{proof}

\def\vsp{\vskip10pt}
\begin{figure}[b]
\scalebox{1.1}{\includegraphics{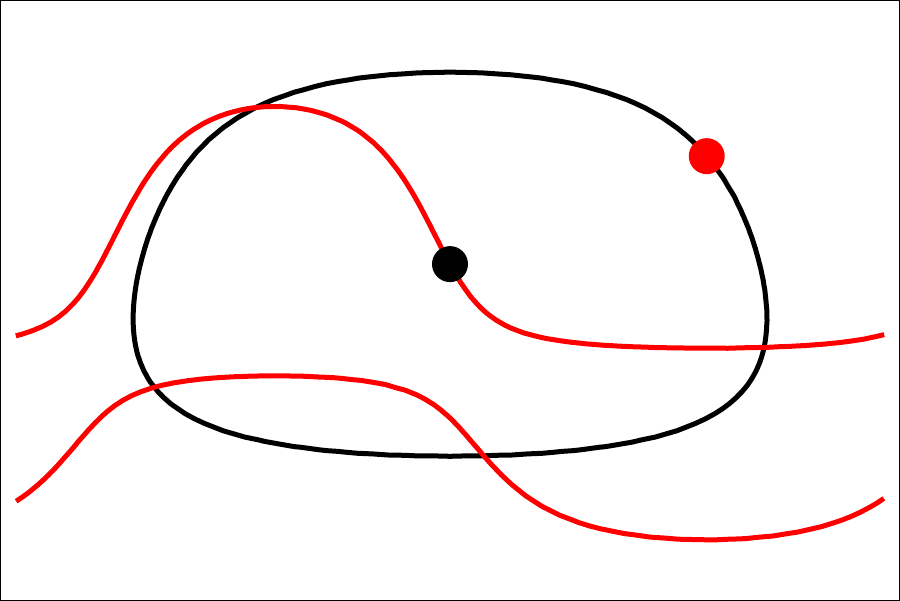}}
\caption{The black (resp., red) point is correctly separated from all
  points on the black (resp., red) curve. The two points cannot belong
  to a SIC set because there are only four remaining points correctly
  separated from them both.}  \vsp
\end{figure}

\begin{figure}[b]
\scalebox{1.1}{\includegraphics{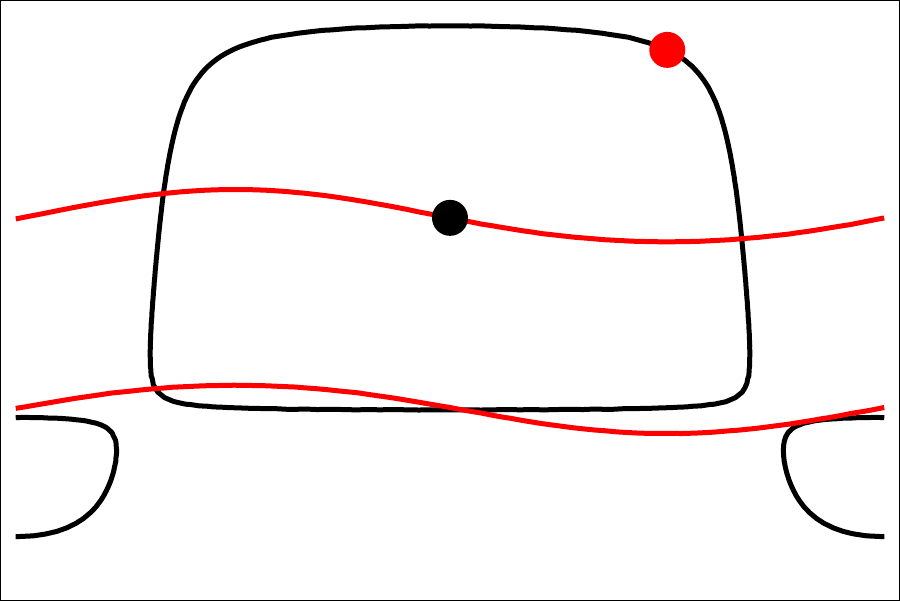}}
\caption{The points that are correctly separated from the black point
  can form a disconnected set. Here, there appear to be five points
  correctly separated from the red and black points, but these five
  points do not in fact form part of a SIC set.}  \vsp
\end{figure}

We are now in a position to illustrate the problem of finding SIC sets
that contain $[\z_1]$ and $[\z_2].$ We can (and shall) assume that a
third point of $\SS$ is $\z[0,\th]$ for some fixed
$\th\in(-\pi/2,\pi/2).$ This point corresponds to one on the central
vertical axis of the rectangle $\sR,$ and will be displayed by a black
dot in the figures. We shall draw some curves to illustrate the
concept of correctly separated points in $\sR,$ meaning that the
distance between the points they represent in $\CP^2$ equals $2\pi/3.$
A fourth point $\z[\si,\phi]$ of $\SS$ will be displayed by a red dot.

\begin{figure}[t]
\scalebox{1.1}{\includegraphics{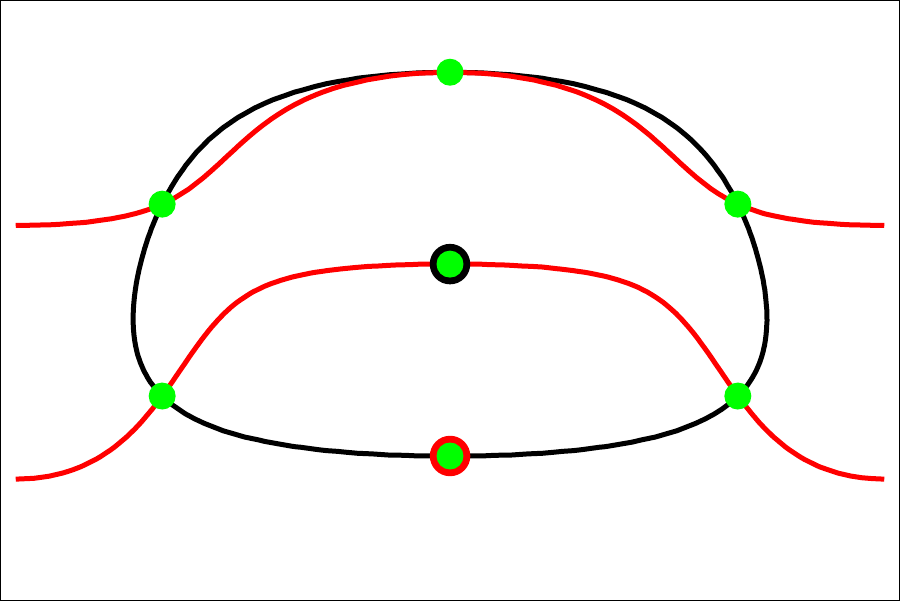}}
\caption{Here the fourth point $\z[0,\frac1{16}\pi-\frac13\pi]$ belongs
  to $\SSt$ which is generated by the remaining five points on the
  intersection of the black and red curves.} \vsp
\end{figure}

\begin{figure}[t]
\scalebox{1.1}{\includegraphics{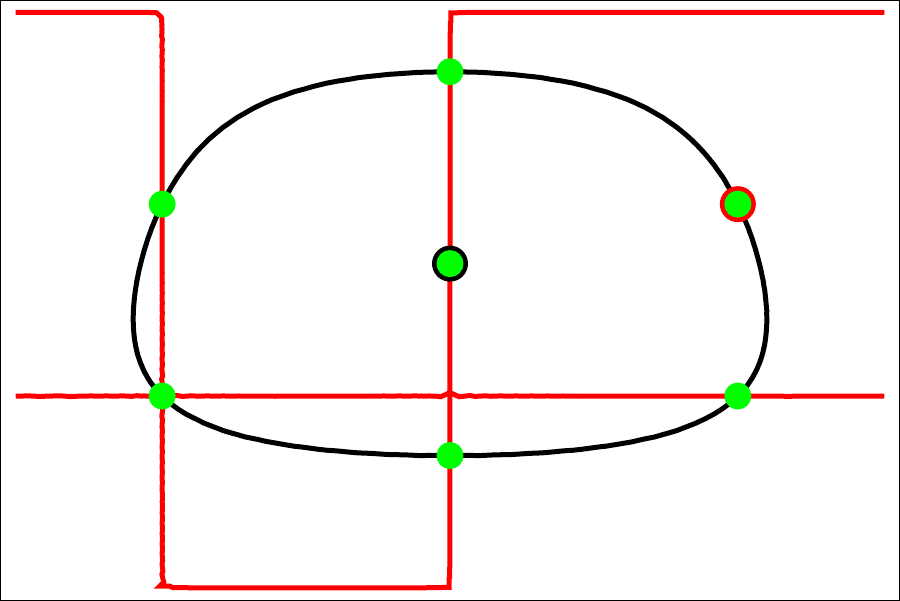}}
\caption{Here the fourth point $\z[\frac23\pi,\sa\pi]\in\SSt$ is
  equidistant from all points on the circles $C_-,C_0,C_2$,
  represented by straight lines in $\sR.$ The segments top and bottom
  collapse to the pinch point.} \vsp
\end{figure}

In Figure~4, $\th=\pi/16$ so that the third point $\z[0,\th]$ of $\SS$
is close to centre of $\sR.$ The black curve is the set of points
$\z[\si,\phi]$ which are a distance $2\pi/3$ from $\z[0,\th].$ The
remaining six points of $\SS$ must therefore lie on this curve. One
such example is represented by the red dot, which actually has
$\phi=\pi/4.$ Points $\z[\tau,\psi]$ a distance $2\pi/3$ apart from
this red point are those on the red curve (which has two
components). The intersection of the black and red curves consists of
points which are correctly separated from both the third and fourth
points. Since there are only four of these (we require five), the
value $x=\tan\phi=1$ cannot in fact occur when $\th=\pi/16.$

The nature of the black curve is heavily dependent on the value
  chosen of $\th$ and $\tt=\tan\th.$ If $\z[\pi,\phi]$ is correctly
  separated from $\z[0,\th]$ and $x=\tan\phi$ then
\begin{equation}\label{pi} 
9(1-3\tt^2)x^2+24\tt x+9\tt^2+5=0.
\end{equation}
Computing the roots of the discriminant as a function of $\tt ,$
\eqref{pi} has distinct roots if and only if
$|\tt|>\sqrt{5/27}=0.430\ldots$ In this case, the black curve has two
connected components, and an example is visible in Figure~5 for which
$\th=\pi/7.$ This time the red point $(\si,\phi)$ is chosen (with $x$
approximately $4.75$) so that there are exactly five points correctly
separated from both the (black and red) third and fourth
points. Subsequent analysis will show that these five points are not
correctly separated from each other.

Although the fourth (red) point in Figure~4 is not admissible (nor, in
fact, is that in Figure~5), Proposition~\ref{7} implies that there
\emph{does} exists a SIC set, namely $\SSt,$ containing the first
three points, so there must be at least six points on the black curve
that \emph{are} admissible. For Figures~6 and 7, we return to the
value $\th=\pi/16,$ and display these six points in green.

In Figure~6, we have chosen the fourth point $\z[\si,\phi]$ to be the
admissible one with $\si=0$ and $\phi$ negative. In Figure~7, we have
chosen the fourth point to be one of the points of $\SSt$ that does
not depend on $\th.$ Recall that the top and bottom boundary of $\sR$
is a single point, and that the horizontal lines $\phi=\pm\pi/6$ are
the circles $C_2,C_3.$ Figure~7 illustrates the fact that any point of
$C_2$ is correctly separated from $[\z_1],$ $[\z_2]$ and a given point
of $C_3,$ as we explained in the proof of Lemma~\ref{pi6}.

\sect{Symmetrization}\label{Sym}

We suppose now that $\SS$ is a SIC set containing, in addition to
$[\z_1]=[0,1,-\om]$ and $[\z_2]=[0,1,-\om^2],$ four more points
$\z[\si_i,\phi_i],$ $i=3,4,5,6,$ with $\tt=\phi_3,$ $x=\phi_4,$
$y=\phi_5,$ $z=\phi_6.$

In view of Theorem~\ref{fa} and equation \eqref{243}, our task is to
investigate the system of polynomial equations given by
\begin{equation}\label{task}
f(x,y,z)=0,\ \ f(\tt,y,z)=0,\ \ f(\tt,x,z)=0,\ \ f(\tt,x,y)=0.
\end{equation}
Since $f$ is itself symmetric, the whole system is invariant under the
action of the group of permutations of $\tt,x,y,z.$ There are
refinements of Buchberger's algorithm for dealing with symmetric
ideals, but we shall adopt the technique outlined in
\cite{St}. Namely, we shall convert the system into a system four
equations, each of which involves only the elementary symmetric
polynomials defined by
\begin{equation}\label{esp}\ba{l} 
a=\tt+x+y+z,\\
b=\tt x+\tt y+\tt z+yz+zx+xy,\\
c=xyz+\tt yz+\tt xz+\tt xy,\\
d=\tt xyz.
\ea\end{equation}
To accomplish this, first define 
\begin{equation}
F_1=f(x,y,z)+f(\tt,y,z)+f(\tt,x,z)+f(\tt,x,y),
\end{equation}
and consider
\begin{equation} 
g(\tt,x,y,z)=\frac{f(\tt,x,y)-f(\tt,x,z)}{y-z}.
\end{equation}
This is a \emph{polynomial} in $\tt,x,y,z,$ so we can symmetrize it to get
\begin{equation}\ba{r} 
F_2=g(\tt,x,y,z)+g(\tt,y,z,x)+g(\tt,z,x,y)\\
   +g(x,y,\tt,z)+g(x,z,\tt,y)+g(y,z,x,\tt).
\ea\end{equation}
Next, set
\begin{equation} 
h(\tt,x,y,z)=\frac{g(\tt,x, y, z)-g(\tt, y, x, z)}{x - y},
\end{equation}
so as to define
\begin{equation}\ba{r} 
F_3= h(\tt,x,y,z) + h(\tt,y,z,x) + h(\tt,z,x,y)\\
+h(x,y,\tt,z) + h(x,z,\tt,y) + h(y,x,\tt,z)\\
+h(z,x,y,\tt) + h(x,y,z,\tt) + h(y,z,x,\tt)\\
 + h(z,y,\tt,x) + h(y,z,\tt,x) + h(z,x,t,y),\\[10pt]
\ds F_4= \frac{h(\tt,x,y,z)-h(x,\tt,y,z)}{\tt- x}.
\phantom{oooooo}
\ea\end{equation}
\medbreak

Each of $F_1,F_2,F_3,F_4$ is a symmetric polynomial in $\tt,x,y,z,$
and can therefore be expressed as a polynomial in $a,b,c,d.$ The proof
of Theorem~\ref{strong} proceeds by examination of the system
\begin{equation}\label{4F}
F_1(a,b,c,d)=0,\ \ F_2(a,b,c,d)=0,\ \ 
F_3(a,b,c,d)=0,\ \ F_4(a,b,c,d)=0.
\end{equation}
To determine the polynomials $F_i$ in practice, we used again the
\emph{Mathematica} command $\mathtt{SymmetricReduction}$. For
completeness we list them explicitly:

{\normalsize\begin{equation*}\ba{l} F_1=36 -
  66 a^2 + 27 a^4 + 218 b - 288 a^2 b + 54 a^4 b + 614 b^2 - 452 a^2
  b^2 + 48 a^4 b^2 + 828 b^3 - 276 a^2 b^3\j + 378 b^4 + 54 b^5 - 41 a
  c + 159 a^3 c - 63 a^5 c - 567 a b c + 270 a^3 b c - 246 a b^2 c +
  18 a^3 b^2 c - 279 a b^3 c\j + 81 a b^4 c + 834 c^2 - 708 a^2 c^2 -
  153 a^4 c^2 + 1968 b c^2 - 171 a^2 b c^2 + 2871 b^2 c^2 - 486 a^2
  b^2 c^2 + 810 b^3 c^2\j + 243 b^4 c^2 - 693 a c^3 - 486 a^3 c^3 +
  2376 a b c^3 - 810 a b^2 c^3 + 567 c^4 - 162 a^2 c^4 + 6399 b c^4 +
  486 b^2 c^4\j + 1701 a c^5 + 2187 c^6 - 712 d + 687 a^2 d + 414 a^4
  d - 2632 b d + 351 a^2 b d + 216 a^4 b d - 4470 b^2 d + 1107 a^2 b^2
  d\j - 5454 b^3 d + 243 a^2 b^3 d - 1782 b^4 d - 486 b^5 d + 453 a c
  d + 927 a^3 c d - 3330 a b c d + 2835 a^3 b c d - 7857 a b^2 c d\j +
  1458 a b^3 c d + 666 c^2 d - 1485 a^2 c^2 d - 4968 b c^2 d + 2268
  a^2 b c^2 d - 24786 b^2 c^2 d - 1944 b^3 c^2 d - 6075 a c^3 d\j -
  9477 a b c^3 d - 1701 c^4 d - 13122 b c^4 d + 4656 d^2 - 531 a^2 d^2
  - 2673 a^4 d^2 + 14436 b d^2 + 9774 a^2 b d^2 + 12042 b^2 d^2\j -
  2349 a^2 b^2 d^2 + 13608 b^3 d^2 + 972 b^4 d^2 + 3861 a c d^2 - 1944
  a^3 c d^2 + 37665 a b c d^2 + 13365 a b^2 c d^2 + 6966 c^2 d^2\j +
  8991 a^2 c^2 d^2 + 7776 b c^2 d^2 + 19683 b^2 c^2 d^2 + 13122 a c^3
  d^2 - 11448 d^3 - 11907 a^2 d^3 - 35640 b d^3 - 12393 a^2 b d^3\j -
  7290 b^2 d^3 - 4374 b^3 d^3 - 16281 a c d^3 - 26244 a b c d^3 -
  13122 c^2 d^3 + 8748 d^4 + 6561 a^2 d^4 + 13122 b d^4.
  \ea\end{equation*}}\\[-15pt] 
{\normalsize\begin{equation*}\ba{l} F_2= 63 a -
  243 a^3 + 81 a^5 + 829 a b - 846 a^3 b + 81 a^5 b + 1706 a b^2 - 642
  a^3 b^2 + 1092 a b^3 + 18 a^3 b^3 + 45 a b^4\j + 81 a b^5 - 1086 c +
  741 a^2 c - 18 a^4 c - 2348 b c + 123 a^2 b c - 153 a^4 b c + 24 b^2
  c - 630 a^2 b^2 c + 2493 b^3 c\j - 486 a^2 b^3 c + 810 b^4 c + 243
  b^5 c + 135 a c^2 + 81 a^3 c^2 - 18 a b c^2 - 486 a^3 b c^2 + 2376 a
  b^2 c^2 - 810 a b^3 c^2\j - 162 c^3 + 567 b c^3 - 162 a^2 b c^3 +
  6399 b^2 c^3 + 486 b^3 c^3 + 1701 a b c^4 + 2187 b c^5 - 1512 a d +
  2583 a^3 d\j + 405 a^5 d - 8709 a b d + 2232 a^3 b d - 11583 a b^2 d
  + 1944 a^3 b^2 d - 7479 a b^3 d - 891 a b^4 d + 1926 c d + 1080 a^2
  c d\j + 1701 a^4 c d + 270 b c d - 5859 a^2 b c d - 2862 b^2 c d +
  4374 a^2 b^2 c d - 11988 b^3 c d - 972 b^4 c d - 2916 a c^2 d\j +
  486 a^3 c^2 d - 25272 a b c^2 d - 7533 a b^2 c^2 d - 5103 a^2 c^3 d
  - 1701 b c^3 d - 8748 b^2 c^3 d - 6561 a c^4 d + 10422 a d^2\j -
  3807 a^3 d^2 + 37071 a b d^2 - 3888 a^3 b d^2 + 29160 a b^2 d^2 +
  4617 a b^3 d^2 + 486 c d^2 + 14337 a^2 c d^2 + 41472 b c d^2\j +
  17010 a^2 b c d^2 + 7290 b^2 c d^2 + 6561 b^3 c d^2 + 15309 a c^2
  d^2 + 26244 a b c^2 d^2 + 13122 c^3 d^2 - 46656 a d^3\j - 6561 a^3
  d^3 - 28431 a b d^3 - 10935 a b^2 d^3 - 10206 c d^3 - 13122 a^2 c
  d^3 - 30618 b c d^3 + 19683 a d^4.\ea\end{equation*}}\\[-15pt]
{\normalsize\begin{equation*}\ba{l} F_3=780 - 978 a^2 - 180 a^4 + 27
  a^6 + 3468 b - 1320 a^2 b - 396 a^4 b + 4968 b^2 - 81 a^2 b^2 + 54
  a^4 b^2 + 2268 b^3\j - 108 a^2 b^3 + 324 b^4 + 243 a^2 b^4 - 594 a c
  - 693 a^3 c - 459 a^5 c + 2250 a b c - 1161 a^3 b c + 6993 a b^2 c -
  1458 a^3 b^2 c\j + 2430 a b^3 c + 729 a b^4 c + 900 c^2 - 1188 a^2
  c^2 - 1458 a^4 c^2 + 648 b c^2 + 7128 a^2 b c^2 - 2430 a^2 b^2 c^2 +
  1701 a c^3\j - 486 a^3 c^3 + 19197 a b c^3 + 1458 a b^2 c^3 + 5103
  a^2 c^4 + 6561 a c^5 - 5304 d + 4446 a^2 d + 4509 a^4 d - 22644 b
  d\j - 6318 a^2 b d + 3645 a^4 b d - 32076 b^2 d - 10044 a^2 b^2 d -
  10692 b^3 d - 486 a^2 b^3 d - 2916 b^4 d + 6642 a c d\j + 4779 a^3 c
  d - 34668 a b c d + 5832 a^3 b c d - 26244 a b^2 c d - 2916 a b^3 c
  d - 7776 c^2 d - 16281 a^2 c^2 d - 76788 b c^2 d\j - 18225 a^2 b c^2
  d - 5832 b^2 c^2 d - 25515 a c^3 d - 26244 a b c^3 d - 26244 c^4 d +
  19440 d^2 + 18954 a^2 d^2 - 2187 a^4 d^2\j + 50868 b d^2 + 19440 a^2
  b d^2 + 75816 b^2 d^2 + 9477 a^2 b^2 d^2 + 5832 b^3 d^2 + 85050 a c
  d^2 + 11664 a^3 c d^2\j + 65610 a b c d^2 + 19683 a b^2 c d^2 +
  20412 c^2 d^2 + 19683 a^2 c^2 d^2 + 78732 b c^2 d^2 - 68040 d^3 -
  39366 a^2 d^3\j - 32076 b d^3 - 13122 a^2 b d^3 - 26244 b^2 d^3 -
  65610 a c d^3 + 26244 d^4.  \ea\end{equation*}}\\[-15pt]
 {\normalsize\begin{equation*}\ba{l}
F_4= 70 a - 168 a^3 + 9 a^5 + 464 a b - 228 a^3 b + 525 a b^2 + 18 a^3 b^2 - 36 a b^3 + 81 a b^4 - 26 c - 399 a^2 c - 153 a^4 c\j
+ 1158 b c - 387 a^2 b c + 2655 b^2 c - 486 a^2 b^2 c + 810 b^3 c + 243 b^4 c - 504 a c^2 - 486 a^3 c^2 + 2376 a b c^2 - 810 a b^2 c^2\j
+ 567 c^3 - 162 a^2 c^3 + 6399 b c^3 + 486 b^2 c^3 + 1701 a c^4 + 2187 c^5 - 762 a d + 963 a^3 d - 3942 a b d + 1215 a^3 b d\j
- 4320 a b^2 d - 162 a b^3 d + 486 c d - 675 a^2 c d - 3348 b c d + 1944 a^2 b c d - 11988 b^2 c d - 972 b^3 c d - 6075 a c^2 d\j
- 6075 a b c^2 d - 1701 c^3 d - 8748 b c^3 d + 4266 a d^2 - 729 a^3 d^2 + 10692 a b d^2 + 3159 a b^2 d^2 + 5994 c d^2\j
+ 3888 a^2 c d^2 + 4374 b c d^2 + 6561 b^2 c d^2 + 6561 a c^2 d^2 - 4374 a d^3 - 4374 a b d^3 - 4374 c d^3.
\ea\end{equation*}}

When $d=0$ (so at least one of $\tt,x,y,z$ vanishes), the expressions
for the $F_i$ simplify greatly, and explicit solutions to \eqref{4F}
can be computed. Not all of the solutions are valid because both
\eqref{4F} is only a necessary (not a sufficent) condition on the
variables $\tt,x,y,z.$ The symmetrization process can introduce
solutions that arise when these quantities are not distinct, as in
\eqref{tt} below. Another problem is the ambiguity of sign in the
horizontal coordinate of $\sR,$ and this will result in our method
capturing solutions like that illustrated in Figure~8.

\sect{Conclusion}

We shall use the theory of Gr\"obner bases to analyse the ideal
\begin{equation}
I=\left<F_1,F_2,F_3,F_4\right>
\end{equation}
of the polynomial ring $\R[a,b,c,d].$ In view of Lemma~\ref{pi6}, we
are not interested in solutions to \eqref{4F} for which $\pm\rec$ is a
root of the polynomial
\begin{equation}\label{qr}
g(x)=x^4-ax^3+bx^2-cx+d.
\end{equation}
Equivalently we want solutions for which
\begin{equation}
\ba{rcl} G(a,b,c,d) &=& 81g(\frac1{\sqrt3})g(-\frac1{\sqrt3})\\[8pt]
&=& 1 - 3a^2 + 6b+ 9b^2 - 18ac- 27c^2 + 18d+ 54bd+ 81d^2
\ea\end{equation} is non-zero. Nor are we interested in solutions of
\eqref{4F} that give rise to a repeated root of \eqref{qr}, for these
can be ignored thanks to Lemma~\ref{LR}.

Using the notion of quotient ideal (see Cox \etal \cite[Chap.~4,
  \S4]{CLO}) we compute the quotient $I:\left<G\right>.$ This is done
by finding a Gr\"obner basis for
\begin{equation}
J=\left<uF_1,uF_2,uF_3,uF_4,(1-u)G\right>
\end{equation}
using a lexicographic ordering with the dummy variable $u$ first in
the dictionary.  Those basis elements that do not involve $u$ are
necessarily divisible by $G$ and provide a basis for the quotient. The
order of the remaining variables is also important, and we used the
\emph{Mathematica} command
$\mathtt{gb:=GroebnerBasis[J,\{u,a,c,b,d\}]}.$ The first element
$\mathtt{gb[[1]]/G}$ equals
\begin{equation}
-(d-1)^3(3 d-1)^3(3+b+3d)(9d-1)^3(1+3b+9d)^3(19+9b+27d),
\end{equation}
and thus we obtain

\begin{theorem}\label{key5} 
Let $\SS=\{[\z_i]\}$ be a SIC set satisfying \eqref{zz12}. Let
$t,x,y,z$ be the `vertical tangents' of $[\z_3],[\z_4],[\z_5],[\z_6],$
and define $b,d$ as in \eqref{esp}. If $\SS$ is not isometric to a
midpoint solution then one or more of the following equations must
hold:
\begin{equation}
d=1,\ d=\fr13,\ d=\fr19,
\ b=-(3d+3),\ b=-\fr13(9d+1),\ -\fr19(27d+19).
\end{equation}
\end{theorem}

We shall examine each possibility in turn.

\smallbreak\noindent \textbf{Case (i).} $d=1/9.$ If $I'$ denotes the
ideal $I$ with $9d-1$ adjoined (in practice, we can merely set
$d=1/9$), one repeats the procedure to determine a basis of
$I':\left<G\right>.$ The new second element $\mathtt{gb[[2]]/G}$
equals
\[ (12b+8-27c^2)(3b+2)(3b+10)(9b+22).\]
First suppose that $b=(27c^2-8)/12.$ This leads to $a+3c=0$ and the
quartic \eqref{qr} has a pair of double roots
\begin{equation}
   x=\fr1{12}(-9c\pm\sqrt{48+81c^2}).
\end{equation}
Expressed more simply, the roots are 
\begin{equation}\label{tt}
x=t,\ t,\ -\fr1{3t},\ -\fr1{3t},
\end{equation}
and we can ignore this solution in view of Lemma~\ref{LR}.

If $b=-2/3$ we get $a=c=0$ and all the roots of \eqref{qr}
are $\pm\rec.$
If $b=-10/3$ we get $a=0$ and $c=\pm8/(3\sqrt3);$ one root
of \eqref{qr} is still $\pm\rec.$
If $b=-22/9,$ we have an instance of Case (iv) in which $a=0$ and
\begin{equation}\label{97}
   c=\pm\fr8{27}\sqrt{\smash{26\pm2\sqrt{97}}\vphantom{I^I}}.
\end{equation}
Provided we take a minus sign inside the square root, \eqref{qr} has
four distinct roots, and provides the `fake SIC set' discussed below
and illustrated in Figure~8.

\smallbreak\noindent\noindent\textbf{Case (ii).} Setting $1+3b+9d=0$
and re-evaluating the quotient ideal forces $d$ to equal one of
$1,1/3,1/9.$ The first leads to
\[a=0,\ b=-\fr{10}3,\ c=0,\ d=1,\]
giving roots of \eqref{qr} that are repeated and include
$\pm\rec.$  The case $d=1/3$ produces no new solutions.

\smallbreak\noindent\noindent\textbf{Case (iii).} $3+b+3d=0.$ This
leads to the solutions
\[a=\pm\fr8{\sqrt3},\ b=-6,\ c=0,\ d=1\]
and
\[a=\pm\fr8{\sqrt3},\ b=0,\ c=0,\ d=-1.\]
In the former case, $\pm\rec$ is still a root of \eqref{qr}.
In the latter case, the quartic has two non-real roots.

\smallbreak\noindent\noindent\textbf{Case (iv).} $19+9b+27d=0.$ This
is in some sense the generic case. It leads to 
\begin{equation}\label{144} \ba{r}
16+9a^2+27ac-144d=0,\hskip200pt\y 4194304 - 73728 a^2 - 132192 a^4 +
6561 a^6 - 4866048ac -746496 a^3c\\ + 78732 a^5 c - 8626176 c^2-
699840 a^2 c^2 + 354294 a^4 c^2 + 1679616 a c^3\\ + 708588 a^3 c^3 +
1889568 c^4 +531441 a^2 c^4=0, \ea\end{equation} giving rise to a
one-parameter family of solutions to \eqref{4F}. To describe this
family, we fix $\tt=\tan\th$ exactly as we did in the
figures of Section~\ref{Sym}. We set

\begin{equation}\label{a2p}
a=\tt+p,\ \ b=\tt p+q,\ \ c=\tt q+r,\ \ d=\tt r
\end{equation}
(the notation is as in Theorem~\ref{fa}), and compute a Gr\"obner
basis of the ideal $K$ generated by $19+9b+27d$ and the left-hand
sides of \eqref{144} in terms of $t.$ This can be accomplished with
the \emph{Mathematica} command\\ 
\hphantom{ooo}$\mathtt{GroebnerBasis[K,\{r,q,p\},
    CoefficientDomain\to RationalFunctions]}.$\\ Provided $t\ne0$ and
$|t|\ne\rec,$ the leading terms are $p,q,r^6.$ This means that the
non-leading monomials are $1,r,r^2,r^3,r^4,r^5,$ and that there exist
six solutions over $\C$ counting multiplicity \cite{Sturm}.

\begin{proof}[Completion of the proof of Theorem~\ref{strong}]
Let us first summarize the argument so far. The existence of six
correctly-separated points in $\CP^2,$ including the ones
$[\z_1],[\z_2]$ we fixed from the start of Section~\ref{2pH} onwards,
leads to a solution of the system \eqref{4F}. Lemma~\ref{pi6} allows
us to dispense with cases in which one root $t$ of \eqref{qr} (or, one
root of $x^3-px^2+qx-r=0$) equals $\pm\rec;$ such cases give rise to
SIC sets isometric to $\SSt.$

Theorem~\ref{key5} provides conditions for any extra solutions, and we
are led to focus on Case (iv), which does supply a family of solutions
to \eqref{4F}. We must show that these do not harbour an undetected
SIC set. In accordance with \eqref{U1}, we can assume that a third
point of $\SS$ equals $\z[0,\th]$ and apply Lemma~\ref{LR}. The
remaining six points of a SIC set would give rise to $\binom63=20$
solutions for each fixed $x.$ But Case (iv) provides at most six sets
of roots.
\end{proof}

We can now be certain that the solutions in Case (iv) are not SIC
sets. For any given rational value of $t,$ the solutions are roots of
polynomials whose coefficients are known exactly. Experimentally, the
number of real solutions varies from two to five according to the
following table:
\begin{equation*}
\ba{|c|c|}\hline
\hbox{real solutions} & \hbox{range}\\\hline\hline
2 & \hphantom{0.11<}|t|<0.1898\\\hline
3 & \quad 0.1899<|t|<0.4386\quad\\\hline
5 & 0.4387<|t|<\rec\\\hline
4 & \rec<|t|<1.1546\\\hline
3 & 1.1547<|t|\hphantom{<0.11}\\\hline
\ea
\end{equation*}

\noindent Although not SIC sets, these solutions validate \eqref{4F}
by virtue of `cross-field passes' of the type described below. Their
changing number as $|t|$ increases reflects the transitional nature of
the curves displayed in Figures~4 to 7.

\begin{exa}\label{fake}\rm
Take $(a,b,c,d)=(0,-22/9,c,1/9)$ where $c$ is given by \eqref{97} with
both minus signs. Then \eqref{qr} becomes
\begin{equation}
 27x^4-66x^2+8\sqrt{\smash{26-2\sqrt{97}}\vphantom{I^I}}x+3=0,
\end{equation}
and has four real roots, namely $x_3=t=-1.687\ldots$ and
\begin{equation}\label{456}
x=x_4=-0.109\ldots\quad y=x_5=0.442\ldots\quad z=x_6=1.354\ldots
\end{equation}

Let $\phi_i=\arctan x_i.$ Set $\si_3=0,$ but for $i=4,5,6$ choose
$\si_i>0$ such that $\z[\si_i,\phi_i]$ is correctly separated from
$\z[0,\phi_3].$ Then $[\z_1],[\z_2],\z[0,\phi_3]$ are
all a distance $2\pi/3$ from each of the six
points $\z[\pm\si_i,\phi_i]$ for $i=4,5,6.$ Moreover, the pairs 
\begin{equation}\label{last6}\ba{lll}
\z[\si_4,\phi_4],\ \z[\si_5,\phi_5]; &
\z[\si_4,\phi_4],\ \z[\si_6,\phi_6]; &
\z[\si_5,\phi_5],\ \z[-\si_6,\phi_6];\\[3pt]
\z[-\si_4,\phi_4],\z[-\si_5,\phi_5];\quad &
\z[-\si_4,\phi_4],\z[-\si_6,\phi_6];\quad &
\z[-\si_5,\phi_5],\ \z[\si_6,\phi_6] \ea\end{equation} are a distance
$2\pi/3$ apart. This does not contradict Lemma~\ref{LR} because
$\z[-\si_i,\phi_i]$ and $\z[\si_i,\phi_i]$ are not correctly
separated. All together, we have constructed nine points in $\CP^2$
for which 27 of the $\binom92$ pairs are correctly separated, though
the resulting configuration is less symmetrical than that of
Example~\ref{eigen}. The seven points $\z[\pm\si_i,\phi_i]$ are shown
in Figure~8; distinguishing a different root $x_3$ from the list
\eqref{456} would give a different picture of the same phenomenon.
\end{exa}

\begin{figure}[t]
\scalebox{1.3}{\includegraphics{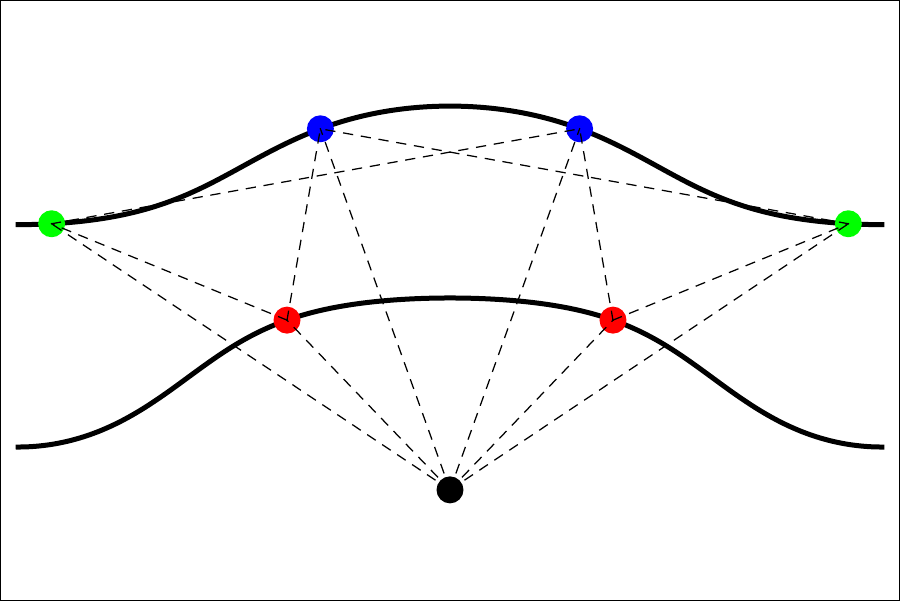}}
\caption{A representation of seven points in $\CP^2$ all a distance
  $2\pi/3$ from $[\z_1],[\z_2].$ Pairs joined by dashed edges are also
  $2\pi/3$ apart.}
\vskip20pt
\end{figure}

\vskip20pt
\normalsize
\section*{Acknowledgments} 

LPH acknowledges support from the Fields Institute, Ontario, the
Perimeter Institute, Ontario, and Tulane University, New Orleans, and
wishes to thank S.~Abramsky, D.~Appleby, R.~Blume-Kohout, D.~Brody,
H.~Brown, S.~Flammia, C.~Fuchs, L.~Hardy, and H.~Zhu for stimulating
discussions. SMS acknowledges support arising from visits to the
University of Nijmegen, the University of Sofia and the University of
Turin, and thanks N.~Lora Lamia for helpful comments. The authors are
grateful to J.~Armstrong for suggesting the use of quotient ideals and
their computation, which improved an earlier approach involving a
count of multiplicities of known solutions.

\bibliographystyle{plain}
\bibliography{hex_2015}

\vfill

\flushleft
\def\yo{\\[1pt]}

Lane Hughston\yo
Department of Mathematics, Brunel University London, Uxbridge UB8~3PH, UK\yo
Department of Mathematics, University College London, London WC1E~6BT,
UK\yo 

St Petersburg State University of Information Technologies, Mechanics
and Optics,\\ Kronwerkskii ave 49, 199034 St Petersburg, Russia\yo
{\it E-mail:} lane.hughston@brunel.ac.uk

\vskip10pt

Simon Salamon\yo 
Department of Mathematics, King's College London, Strand, London
WC2R~2LS, UK\yo 
{\it E-mail:} simon.salamon@kcl.ac.uk

\enddocument